\documentclass[12pt]{article}
\usepackage{amsfonts}
\usepackage{amsrefs}
\usepackage{amsmath}
\usepackage{amsfonts}
\usepackage{amsthm}
\usepackage{enumerate}
\usepackage{amsrefs}
\usepackage{geometry}
\ifdefined\directlua
\usepackage{fontspec}
\usepackage{microtype}
\fi

\newcommand{\cP}{\mathcal P}
\newcommand{\cF}{\mathcal F}

\newcommand{\KK}{\mathbb K}

\newcommand{\cod}{\operatorname{cod}}
\newcommand{\Res}{\operatorname{Res}}

\newcommand{\cG}{\mathcal G}

\newcommand{\cS}{\mathcal S}
\newcommand{\cQ}{\mathcal Q}
\newcommand{\cL}{\mathcal L}

 \newcommand{\cC}{\mathcal C}

\newcommand{\Rad}{\mathrm{Rad}\,}

\newcommand{\PG}{\mathrm{PG}}
\newcommand{\GL}{\mathrm{GL}}

\newcommand{\rank}{\mathrm{rank}\,}
\newcommand{\ext}{\mathrm{Ext}}
\newcommand{\expa}{\mathrm{Exp}}
\newcommand{\dec}{\mathrm{Dec}}
\newcommand{\chr}{\mathrm{char}\,}

\newtheorem{theorem}{Theorem}[section]
\newtheorem{theo}{Theorem}
\newtheorem{lemma}[theorem]{Lemma}
\newtheorem{prop}[theorem]{Proposition}

\newtheorem{corollary}[theorem]{Corollary}

\makeatletter
\renewcommand\appendix{\par
  \setcounter{section}{0}%
  \setcounter{subsection}{0}%
  \gdef\thesection{\@Alph\c@section}
 \renewcommand\section{\@startsection {section}{1}{\z@}%
                                   {-3.5ex \@plus -1ex \@minus -.2ex}%
                                   {2.3ex \@plus.2ex}%
                                   {\normalfont\Large\bfseries Appendix }}
}
\makeatother

\theoremstyle{remark}
\newtheorem{remark}{Remark}

\author{Ilaria Cardinali, Luca Giuzzi}
\title{Geometries arising from trilinear forms on low-dimensional vector spaces}
\begin{document}
\maketitle

\begin{abstract}
  Let $\cG_k(V)$ be the $k$-Grassmannian of a vector space $V$ with
  $\dim V=n$.
 Given a hyperplane $H$ of $\cG_k(V)$, we define in~\cite{ILP17} a point-line subgeometry of $\PG(V)$ called the {\it geometry of poles of $H$}.   In the present paper, exploiting the classification of alternating
  trilinear forms in low dimension, we characterize the possible geometries of poles
  arising for $k=3$ and $n\leq 7$ and propose some new constructions.
  We also   extend a result of~\cite{DS10}
  regarding the existence of line spreads of $\PG(5,\KK)$ arising from hyperplanes of $\cG_3(V).$
\end{abstract}
\section{Introduction}\label{Introduction}
Denote by $V$ a $n$-dimensional vector space over a field $\KK$.
For any fixed $1\leq k < n$ the $k$-Grassmannian  ${\cG}_k(V)$ of $V$ is
the point-line geometry whose points are the $k$--dimensional vector subspaces of $V$ and whose lines are the sets $\ell_{Y,Z}:=\{X~\colon ~Y\subset X\subset Z, ~\dim X=k\}$ where $Y$ and $Z$ are subspaces of $V$
with $Y\subset Z,$ $\dim Y=k-1$ and $\dim Z=k+1$. Incidence is containment.

It is well known that the geometry $\cG_k(V)$ affords a full projective
embedding, called  Grassmann (or Pl\"ucker)
embedding and denoted by  $\varepsilon_k$, sending every $k$-subspace
$\langle v_1,\ldots,v_k\rangle$ of $V$ to the point $[v_1\wedge\cdots\wedge v_k]$
of $\PG(\bigwedge^kV)$, where we adopt the notation $[u]$ to refer
to the projective point represented by the vector $u$.
Also, for $X\subseteq V$ we put $[X]:=\{ [x] : x\in\langle X\rangle \}$
 for the points of the projective space induced by $\langle X\rangle$.

A hyperplane $H$ of $\cG_k(V)$ is a proper subspace of $\cG_k(V)$ such
that any line of $\cG_k(V)$ is either contained in $H$ or it intersects
$H$ in just one point.

It is well known, see Shult \cite{shult92} and also Havlicek \cite{H81}, Havlicek and Zanella \cite{HZ08}
and De Bruyn \cite{B09},
that the hyperplanes of $\cG_k(V)$ all arise from the Pl\"ucker embedding of $\cG_k(V)$
in $\PG(\bigwedge^kV)$, i.e. they bijectively correspond to proportionality
classes of non-zero linear functionals on $\bigwedge^kV$.
More in detail, for any hyperplane $H$ of $\cG_k(V)$ there is a non-null linear functional
$h$ on $\bigwedge^kV$ such that $H=\varepsilon_k^{-1}([\ker(h)]\cap \varepsilon_k(\cG_k(V)))$.
Equivalently, if $\chi_h:V\times\ldots\times V\rightarrow \KK$ is the alternating $k$-linear form on $V$ associated to the linear functional $h$, defined by the clause $\chi_h(v_1,\ldots, v_k) := h(v_1\wedge\cdots\wedge v_k)$, then the hyperplane $H$ is the set of the $k$-subspaces of $V$ where $\chi_h$ identically vanishes. So, the hyperplanes of $\cG_k(V)$ bijectively correspond also  to proportionality classes of non-trivial alternating $k$-linear forms of $V$.

In a recent paper \cite{ILP17} we introduced the notion of \emph{$i$-radical}
of a hyperplane $H$ of $\cG_k(V)$. In the present work we
shall just consider the case of the \emph{lower radical}  $R_{\downarrow}(H)$, for $i=1$,
and that of the \emph{upper radical} $R^{\uparrow}(H)$, for $i=k-1$.
The lower radical $R_{\downarrow}(H)$ of $H$ is the set of points
$[p]\in\PG(V)$ such that all $k$-spaces $X$ with $p\in X$ belong to $H$;
the upper radical $R^{\uparrow}(H)$ of $H$ is the set of $(k-1)$-subspaces
$Y$ of $V$ such that all $k$-spaces through $Y$ belong to $H$.

In the same paper~\cite{ILP17} we investigated the problem of determining
under which conditions the upper radical of a given hyperplane might be
empty. Working in the case $k=3$, we also defined a point-line subgeometry
$\cP(H)=(P(H),R^{\uparrow}(H))$ of $\PG(V)$ called \emph{the geometry of poles} of $H$, whose points
are called $H$-poles, a point $[p]\in P(H)$ and a line $\ell\in R^{\uparrow}(H)$ being incident
 when $p\in\ell$ (see Section~\ref{geom of poles}).
Some aspects of this geometry had already been studied,
under slightly different settings, in \cite{DS10,DS14}.
It has been shown in~\cite{DS10} that the set $P(H)$ is actually either
$\PG(V)$ or an algebraic hypersurface in $\PG(V)$; for more details see
Section~\ref{geom of poles}.

In this paper  we shall focus on the geometry $\cP(H)$, providing explicit equations for its points and lines and also a geometric  description in
the  cases where a complete
classification of trilinear forms on a vector space $V$ is available, namely $\dim(V)\leq6$ with $\KK$ arbitrary (see~\cite{Revoy79}) and $\dim(V)=7$ with $\KK$ a perfect field with cohomological dimension at most $1$ (see~\cite{CH88}).

We shall briefly recall the definition of  geometry of poles in the next section and we will state our main results in
 Section~\ref{main results}. The organization of the paper is outlined in Section~\ref{organization}.

\subsection{The geometry of poles}\label{geom of poles}
Assume $k=3$ and let $H$ be a given hyperplane of $\cG_3(V)$. For any  (possibly empty) projective space
$[X]$ we shall denote by $\dim(X)$ the vector  dimension of $X.$

Let $[p]$ be a point of $\PG(V)$ and consider the point-line geometry ${\cS}_p(H)$  having as points, the set of lines of $\PG(V)$ through $[p]$ and as lines, the set of planes $[\pi]$ of $\PG(V)$ through $[p]$ with
$\pi\in H$.
It is easy to see (see~\cite{ILP17}) that ${\cS}_p(H)$ is a polar space of symplectic type (possibly a trivial one). Let $R_p(H) := \Rad({\cS}_p(H))$ be the radical of ${\cS}_p(H)$ and put $\delta(p):=\dim(\Rad({\cS}_p(H))).$

We call $\delta(p)$ the \emph{degree} of $[p]$ (relative to $H$). If $\delta(p) = 0$ then we say that $[p]$ is \emph{smooth}, otherwise we call $[p]$ a \emph{pole} of $H$ or, also, a $H$-\emph{pole} for short. Clearly, a point is a pole if and only if it belongs to a line of the upper radical  $R^\uparrow(H)$ of $H$. So, $R^\uparrow(H) = \emptyset$ if and only if all points are smooth.

As the vector space underlying the symplectic polar space ${\cS}_p(H)$ has dimension $n-1$, $\delta(p)$ is even when $n$ is odd and it is odd if $n$ is even. In particular, when $n$ is even  all the points are poles of degree at least $1$. If they all have degree $1$, then we say that $H$ is \emph{spread-like}.
%In both cases, the poles of degree $\delta(p) = n-1$ are just the points of $R_\downarrow(H)$.

We shall provide in Theorem~\ref{ts} a direct geometric proof
of the result of \cite[Theorem 4]{ILP17}, stating
that for $n=6$ there are spread-like hyperplanes if and only if the field $\KK$
is not quadratically closed, extending a result of \cite{DS14}; this is also implicit in
 the classification of~\cite{Revoy79};
observe that for $n=8$ we prove
 in \cite{ILP17} that for  quasi-algebraically closed fields there are no spread-like hyperplanes.

More in general, when all points of $\PG(V)$ are poles of the same
degree $\delta$ and $(\delta+1)|n$, the set $\{\pi_p: [p]\in\PG(V) \}$ of all
subspaces $\pi_p:=\{ [u]\in [\ell] : p\in \ell \mbox{ and } \ell\in R^{\uparrow}(H) \}$
as $[p]$ varies in $\PG(V)$ might in some case possibly be
a spread of $\PG(V)$ in spaces of projective dimension $\delta$.
We are not currently aware of any case where this happens for $\delta>1$
and we propose this as
a problem which might be worthy of further investigation.

\subsection{Main results}\label{main results}
Before stating our main results we need to fix a terminology for linear functionals on $\bigwedge^3V$ and to recall what is currently known from the literature regarding their classification. As already pointed out, a classification of alternating trilinear forms of $V$ would determine a classification of the geometries of poles defined by hyperplanes of $\cG_3(V).$

We will first introduce the notion of isomorphism for hyperplanes and the notions of equivalence, near equivalence and geometrical equivalence for $k$-linear forms in general.

 We say that two hyperplanes $H$ and $H'$ of $\cG_k(V)$ are \emph{isomorphic}, and we write $H \cong H'$, when $H' = g(H) := \{g(X)\}_{X\in H}$ for some $g\in\GL(V)$, where $g(X)$ is the natural action of $g$ on the subspace $X$, i.e. $g(X)=\langle g(v_1),\dots, g(v_k)\rangle$ for $X=\langle v_1,\dots, v_k\rangle.$

  Recall that two alternating $k$-linear forms $\chi$ and $\chi'$ on $V$ are said to be \emph{(linearly) equivalent} when
\[\chi'(x_1,\ldots, x_k) ~ = ~ \chi(g(x_1),\ldots, g(x_k)), ~~~ \forall x_1,\ldots, x_k \in V\]
for some $g\in\GL(V)$. Accordingly, if $H$ and $H'$ are the hyperplanes associated to $\chi$ and $\chi'$, we have $H\cong H'$ if and only if $\chi'$ is proportional to a form equivalent to $\chi$. Note that if $\chi' = \lambda\cdot \chi$ for a scalar $\lambda \neq 0$ then $\chi$ and $\chi'$ are equivalent if and only if $\lambda$ is a $k$-th power in $\KK$.

We say that two forms $\chi$ and $\chi'$ are \emph{nearly equivalent}, and we write $\chi\sim \chi'$, when each of them is equivalent to a non-zero scalar multiple of the other. Hence  $H\cong H'$ if and only if $\chi\sim\chi'$.

We extend the above terminology to linear functionals of $\bigwedge^kV$ in a
natural way, saying that two linear functionals $h, h' \in (\bigwedge^kV)^*$ are \emph{nearly equivalent} and writing $h\sim h'$ when their corresponding $k$-alternating forms are nearly equivalent.

We say that two hyperplanes $H$ and $H'$ are \emph{geometrically equivalent} if the incidence graphs of their geometries of poles are isomorphic; the forms defining geometrically equivalent hyperplanes are called \emph{geometrically equivalent}  as well.

Note that nearly equivalent forms are always geometrically equivalent but the converse  does not hold in general.
For example, let $V$ be  a vector space over a field $\KK$ which is not quadratically
closed and suppose $\dim(V)=6$. To any quadratic extension of $\KK$ there correspond a
Desarguesian line-spread $\cS$ of $\PG(V)$, and the geometry $(\PG(V),\cS)$ is a geometry
of poles associated to a trilinear form. All hyperplanes inducing line-spreads are
geometrically equivalent.
If $\KK$ is a finite field or $\KK={\mathbb R}$, it is easy to see that the hyperplanes
inducing $\cS$ must also be isomorphic. However, this is not the case when $\KK={\mathbb Q}$
or when $\KK$ is a field of characteristic $2$ which is not perfect. In particular, in
the latter case hyperplanes arising from forms of type $T_{10,\lambda}^{(1)}$ and
$T_{10,\lambda}^{(2)}$, see Table~\ref{Tab F}, are geometrically equivalent but not
isomorphic.

Clearly, nearly equivalent or isomorphic hyperplanes are always  geometrically equivalent.

\subsubsection{Types for linear functionals of $\bigwedge^3V$}\label{prel k-linear alt}
Given a non-trivial linear functional $h\in(\bigwedge^3 V)^*$, let $\chi_h$ and $H_h$ be respectively the alternating trilinear form and the hyperplane of ${\cG}_3(V)$ associated to it. When no ambiguity might arise, we shall feel free to drop the subscript $h$ in our notation.

By definition, $R_{\downarrow}(H) = [\Rad(\chi)]$, where $\Rad(\chi)=\{v\in V\colon \chi(x,y,v)=0,\forall\, x,y\in V\}.$
Define the \emph{rank} of $h$ as $\rank(h) := \cod_V(\Rad(\chi))=\dim (V/\Rad(\chi))$.
Obviously,
functionals of different rank can never be nearly equivalent.

It is known that if $h$ is a non-trivial trilinear form, then $\rank(h) \geq 3$ and $\rank(h) \neq 4$ (see~\cite[Proposition 19]{ILP17} for the latter result).

Fix now a basis $E:=(e_i)_{i=1}^n$ of $V$. The dual basis of $E$ in $V^*$ is  $E^* := (e^i)_{i=1}^n$, where
$e^i\in V^*$ is the linear functional such that $e^i(e_j) = \delta_{i,j}$ (Kronecker symbol).
The set $(e^{i}\wedge e^{j}\wedge e^{k})_{1\leq i<j<k\leq n}$ is the basis of $(\bigwedge^3V)^*$ dual of the basis $(e_i\wedge e_j\wedge e_k)_{1\leq i < j < k\leq n}$ of $\bigwedge^3V$ canonically associated to $E$.
 %We will refer to $(e^{i}\wedge e^{j}\wedge e^{k})_{1\leq i<j<k\leq n}$ and to $(e_i\wedge e_j\wedge e_k)_{1\leq i < j < k\leq n}$ as the canonical bases of  $(\bigwedge^3V)^*$ and $(\bigwedge^3V)$ respectively.
We shall adopt the convention of writing $\underline{ijk}$ for $e^i\wedge e^j\wedge e^k$, thus representing linear functionals of $\bigwedge^3V$ as linear combinations of symbols like $\underline{ijk}$.

In Table~\ref{Tab F}, see Appendix~\ref{appendix}, we list a number of possible types
of linear functionals of $\bigwedge^3V$ of rank at most $7$, denoted by the symbols
$T_1,\ldots, T_9$ and $T^{(1)}_{10,\lambda}$, $T^{(2)}_{10,\lambda}$, $T^{(1)}_{11,\lambda}$ and $T^{(2)}_{11,\lambda}$,
$T_{12,\mu}$ where $\lambda$ is a scalar subject to the conditions specified in the table.
Whenever $T$ is one of the types of Table~\ref{Tab F}, we say that $h\in (\bigwedge^3V)^*$ is of \emph{type} $T$ if $h$ is nearly equivalent to the linear functional described at row $T$ of Table~\ref{Tab F}. The \emph{type} of $H_h$ or $\chi_h$ is the type of $h$.
By definition, functionals of the same type are nearly equivalent.
Note that the definitions of each of these types make sense for any $n$ and for
any field $\KK$, provided that $n$ is not smaller than the rank of (a linear functional admitting) that description and that $\KK$ contains elements satisfying the special conditions there outlined.

In particular, Table~\ref{Tab F} provides a complete classification (up to
equivalence) in the case of perfect fields of cohomological dimension at most $1$ and $\dim(V)\leq 7$, see~\cite{CH88,Revoy79}. We recall that in~\cite{CH88} the authors also
determine the full automorphism group associated to each form.

\par

It is well known that a general classification of trilinear forms up to equivalence is hopeless; for instance
for $\KK={\mathbb C}$ and $n=9$ there are infinite families of linearly inequivalent trilinear forms; see e.g.
\cite{DS10}.
 \par
 When $n\leq6$, Revoy~\cite{Revoy79} proves that all trilinear forms, up to equivalence, are of type $T_i$, $1\leq i\leq 4$, $T_{10,\lambda}^{(1)}$, $T_{10,\lambda}^{(2)}$. In particular, if the field is quadratically closed, all forms of rank $6$ are either of type $T_3$ or of type $T_4$.
If the field is not quadratically closed, it is possible
 to distinguish two families of classes of
forms linearly equivalent among themselves according as
they are  equivalent to $T_3$ or to
$T_4$ over the quadratic closure $\KK^{\square}$ of $\KK$.
More in detail, if $\chr(\KK)$ is odd,  $T_{10,\lambda}^{(1)}=T_{10,\lambda}^{(2)}$ and each form of
a type in $T_{10,\lambda}^{(1)}$ is equivalent to $T_3$ over $\KK^{\square}$.
If $\chr(\KK)=2$, then the classes $T_{10,\lambda}^{(2)}$ and $T_{10,\lambda}^{(1)}$ are
in correspondence with respectively the separable and the inseparable quadratic extensions
of the field $\KK$; furthermore, any form of a type in $T_{10,\lambda}^{(2)}$ is equivalent
to a form of type $T_3$ in $\KK^{\square}$, while any form of type $T_{10,\lambda}^{(1)}$ is
equivalent in $\KK^{\square}$ to a form of type $T_4$.
% The classes in these two families are respectively in bijection
% with the set of all the separable quadratic extensions of $\KK$ in the first case
% while in the second case they correspond with $T_4$ if $\chr(\KK)$ is odd or $\KK$ is perfect
%  or they correspond with the  the set of the inseparable quadratic extensions of $\KK$ if $\KK$ is not
%  perfect with $\chr(\KK)=2$.

If $n=7$ and $\KK$ is a perfect field of cohomological dimension at most $1$, Cohen and Helmick~\cite{CH88} show that all trilinear forms, up to equivalence, are of a type described in Table~\ref{Tab F}.

Under the conditions of Table~\ref{Tab F}, we shall provide a classification for the geometries of poles.

We now present our main results; for the notions of extension, expansion and block decomposition
as well as some of the notation, see Section~\ref{sec 3}.
By the symbol $|x,y|_{ij}$ we mean the $(i,j)$-Pl\"ucker coordinate of
the line $[x,y]$ spanned by the vectors $x=(x_i)_{i=1}^n$ and $y=(y_i)_{i=1}^n$ written in coordinates with respect to the basis $E$, i.e. $|x,y|_{ij}:=x_iy_j-x_jy_i$
is the $ij$-coordinate of $e_i\wedge e_j$ with respect to the basis $(e_i\wedge e_j)_{1\leq i<j\leq n}$.

\begin{theo}\label{main theorem 1}
  Suppose $\dim(V)\leq6$ and let $h$ be a non-trivial linear functional on $\bigwedge^3V$ having type as described in Table~\ref{Tab F}, with associated alternating trilinear form $\chi$. Denote by $H$ the
  hyperplane of $\cG_3(V)$ defined by $h$.
  Then one of the following occurs:
\begin{enumerate}
\item\label{m1p1}
  $h$ has type $T_1$ (rank $3$).  In this case $H$ is the trivial hyperplane centered at $\Rad(\chi)$ and
  $R^{\uparrow}(H)$ is the set of the lines of $\PG(V)$ that meet $[\Rad(\chi)]$ non-trivially.
\item\label{m1p2}
  $h$ has type $T_2$ (rank $5$), namely $\Rad(\chi)$ is $1$-dimensional. In this case $H$ is a trivial extension $\ext_{\Rad(\chi)}(\expa(H_0))$ of a symplectic hyperplane $\expa(H_0)$, constructed in a complement $V_0$ of $\Rad(\chi)$ in $V$ starting from the line-set $H_0$ of a symplectic generalized quadrangle. The elements of $R^\uparrow(H)$ are the lines of $\PG(V)$ that either belong to $H_0$ or pass through the point $[\Rad(\chi)] = R_\downarrow(H)$ or such that their projection onto $V_0$ is in $H_0.$
\item\label{m1p3}
  $h$ has type $T_3$ (rank $6$).  Then $H$ is a decomposable hyperplane $\dec(H_0,H_1)$ arising form the hyperplanes $H_0$ and $H_1$ of $\cG_3(V_0)$ and $\cG_3(V_1)$ for a suitable decomposition $V = V_0\oplus V_1$ with $\dim(V_0) = \dim(V_1) = 3$. Then $R^{\uparrow}(H) = \{[x,y] \colon x\in V_0\setminus\{0\}, ~ y\in V_1\setminus\{0\}\}.$
\item\label{m1p4}
  $h$ has type $T_4$ (rank $6$). Then $R^{\uparrow}(H) =
  \{[x,y]=[a+b,\omega(a)] \colon a\in V_0\setminus\{0\}, b\in V_1\}\cup\{ [x,y]\subseteq V_1\}$
    for a  decomposition $V = V_0\oplus V_1$ with $\dim(V_0) = \dim(V_1) = 3$ and $\omega$
  an isomorphism of $V$ interchanging $V_0$ and $V_1$.
\item\label{m1p5}
  $h$ has type $T^{(1)}_{10,\lambda}$ or $T^{(2)}_{10,\lambda}$ (rank $6$). Then $R^\uparrow(H)$ is a Desarguesian line spread of $\PG(V)$ corresponding
  to the field extension $\KK[\mu]$ with $\mu$ a root of $p_{\lambda}(t)$.
\end{enumerate}
\end{theo}

\begin{theo}\label{ts}
Let $V:=V(6,\KK).$ Line-spreads of $\PG(V)$ induced by hyperplanes of $\cG_3(V)$ exist if and
only if $\KK$ is a non-quadratically closed field.
\end{theo}

Draisma and Shaw prove that when $\KK$ is a finite field there
always exist hyperplanes of $\cG_3(V)$ with $\dim V=6$ having a Desarguesian
spread as upper radical \cite[\S 3.1 and \S 3.2]{DS14},
while these hyperplanes do not exist when $\KK$ is algebraically closed
with characteristic $0$, \cite[Remark 9]{DS14}.

Our Theorem~\ref{ts} generalizes their results to arbitrary fields $\KK$ and provides necessary
and sufficient conditions for the existence of spread-like hyperplanes for $n=6$; its statement
correspond to Theorem 20, point 5 in \cite{ILP17} (and also to Theorem 4 in \cite{ILP17}).

It also further clarifies
the result of~\cite{Revoy79} linking  the forms $T_{10,\lambda}^{(i)}$ with quadratic extensions
of the field $\KK$.

\begin{theo}\label{main theorem 2}
  Suppose $\dim(V)=7$ and let $h$ be a non-trivial linear functional on $\bigwedge^3V$ having type as described in Table~\ref{Tab F}.
   Denote by $\chi$ the associated alternating trilinear form  and by $H$ the
  hyperplane of $\cG_3(V)$ defined by $h$.

If $\rank(h)\leq 6$ then $H$ is a trivial extension $\ext_{\Rad(h)}(H')$ where $H'$ is a hyperplane of $\cG_3(V')$ with $V=\Rad(h)\oplus V'$, $\dim(V')\leq 6$, and $H'$ is defined by a trilinear form $h'$  of type as in Theorem~\ref{main theorem 1}.

If $\rank(h)=7$  then one of the following occurs:
\begin{enumerate}[1)]
\item\label{m2p1}
  $h$ has type $T_5$. Then, there are two
  non-degenerate symplectic polar spaces ${\cS}_1$ and ${\cS}_2$, embedded as distinct hyperplanes in $\PG(V)$ such that both determine the same polar space ${\cS}_0$ on their intersection. The radical of ${\cS}_0$ is a point, say $p_0$. There are also two totally isotropic planes $A_1$ and $A_2$ of ${\cS}_0$ such that $A_1\cap A_2 = \{p_0\}$. The poles of $H$ are the points of $\cS_1\cup\cS_2$, the poles of degree $4$ being the points of $A_1\cup A_2$. The lines of ${\cP}(H)$ are the totally isotropic lines of ${\cS}_i$ that meet $A_i$ non-trivially, for $i = 1, 2$.
\item\label{m2p2}
 $h$ has type $T_6$. The poles of $H$ lie in a hyperplane $\cS$ of $\PG(V)$. A non-degenerate polar space
 of symplectic type is defined in $\cS$ and a totally isotropic plane $A$ of $\cS$ is given.
 The lines of ${\cP}(H)$ are the totally isotropic lines of $\cS$ that meet $A$ non-trivially.
 The points of $A$ are the poles of $H$ of degree $4$.
\item\label{m2p3}
  $h$ has type $T_7$. The poles of $H$ are the points of a cone of $\PG(V)$
having as vertex a plane $A$ and as basis a hyperbolic quadric $\cQ$. A conic $\cC$ is
given in $A$ such that the elements of $\cC$ are the poles of degree $4$.
There is a correspondence mapping
each point $[x]\in\cC$ to a line $\ell_x$ contained in a regulus of $\cQ$.
For each $[x]\in\cC$ it is possible to define a line spread
$\cS_x$ of $\langle A,\ell_x\rangle/\langle x\rangle\cong\PG(3,\KK)$ such that
$R^{\uparrow}(H)=\{\ell\subseteq\langle x,s\rangle: [x]\in\cC, s\in\cS_x\}$.
\item\label{m2p4}
 $h$ has type $T_8$.
 In this case $H=\expa(H_0)$ is a symplectic hyperplane. In particular, the geometry ${\cP}(H)$ is a non-degenerate polar space of symplectic type and rank $3$, naturally embedded in a hyperplane $[V_0]$ of $\PG(V)$. All poles of $H$ have degree $4$.
\item\label{m2p5}
 $h$ has type $T_9$. Then ${\cP}(H)$ is a split-Cayley hexagon naturally embedded in a non-singular quadric of $\PG(V)$. All poles of $H$ have degree $2$.
\item\label{m2p6} $h$ has type $T^{(i)}_{11,\lambda}$, $i=1, 2$.
 The poles of $H$ are the points of a subspace $\cS$ of codimension $2$ in
 $V$. There is only one point $[p]\in\cS$ which is a pole of degree $4$. Furthermore, there is a
 line-spread $\cF$ of $\cS/\langle p\rangle\cong\PG(3,\KK)$ such that
 $R^{\uparrow}(H):=\{\ell\subseteq[\pi,p]: \pi\in\cF \}.$
\end{enumerate}
\end{theo}
Theorems~\ref{main theorem 1} and~\ref{main theorem 2} correspond to
Theorems 20 and 21 of \cite{ILP17}, where they were presented without a detailed proof.
In the present paper we have chosen to refine the results announced \cite{ILP17},
by providing a fully geometric description of the geometries of poles arising for $n\leq 7$,
without having to recourse to coordinates.
In any case, the original statements for cases 3,5 and 7 of \cite[Theorem 21]{ILP17} can be
immediately deduced from Theorem~\ref{main theorem 2} in light of the equations of
Table~\ref{tab_n7}.

\subsection{Organization of the paper}\label{organization}
In Section~\ref{geometry of poles} we shall explain how to algebraically describe points and lines of a geometry of poles.
Draisma and Shaw \cite{DS10} have shown that either the set of $H$-poles
is all of $\PG(V)$ or it determines an
algebraic hypersurface in $\PG(V)$ described by an equation
of degree $(n-3)/2$. In Section~\ref{geometry of poles} we
shall study  such varieties. In particular, in Section~\ref{ssec2} we shall
explicitly determine their equations as determinantal varieties and in Section~\ref{rehy}
describe some hyperplanes whose variety of poles is reducible in the product of distinct linear
factors.
In Section~\ref{sec 3} we will present  three families of hyperplanes of $\cG_k(V)$ obtained by {\it extension}, {\it expansion} and {\it block decomposable construction}.
Our main theorems will be proved in Section~\ref{sec4}.
For the ease of the reader, all the tables are collected in Appendix~\ref{tables}.

\section{Geometry of poles}\label{geometry of poles}
Throughout this section we take $E=(e_i)_{1\leq i\leq n}$ as a given basis of $V$ and the coordinates of vectors of $V$ will be given with respect to $E.$
\subsection{Determination of points and lines}
\label{determination of points and lines}
Let $h\colon \bigwedge^3 V\rightarrow \KK$ be a linear functional associated to a given hyperplane $H$ of $\cG_3(V)$, where $V$ is a $n$-dimensional vector space over a field $\KK.$

For any $u\in V$ consider the bilinear alternating form
\[h_u\colon V/\langle u\rangle\times V/\langle u\rangle\rightarrow \KK,\,\,h_u(x+\langle u\rangle,y+\langle u\rangle):=h(u\wedge x\wedge y).\]
By definition of $H$-pole, a point $[u] \in \PG(V)$ is a $H$-pole if and only if the radical of $h_u$ is not trivial. Consider also the bilinear alternating form on $V$
\begin{equation}\label{eq Mu}
 \chi_u\colon V\times V\rightarrow \KK,\,\,\chi_u(x,y)=h(u\wedge x\wedge y)=x^TM_uy
\end{equation}
where $M_u$ is the matrix associated to $\chi_u$ with respect to the basis $E$ of $V$.
 % $\{(e_i,e_j)\}_{1\leq i\not= j \leq n}$ of $V\times V$.
 Clearly, $\Rad(h_u)=(\Rad(\chi_u))/\langle u\rangle$; thus the rank of the matrix of $h_u$ with respect to any basis of $V/\langle u\rangle$
 and the rank of the matrix $M_u$ are exactly the same.
\begin{prop}\label{rank Mu-prop}
Let $[u]$ be a point of $\PG(V)$ with $u=(u_i)_{i=1}^n$ and let  $M_u$ be the
$n\times n$-matrix associated to the alternating bilinear form $\chi_u$. If $u_i\not=0$ then the $i$-th column (row) of $M_u$ is a linear combination of the other columns (rows) of $M_u.$
%Let $M_u$ be the $n\times n$-matrix associated to the alternating bilinear form $\chi_u$. Then $\rank(M_u)\leq n-1.$
 \end{prop}
\begin{proof}
%  By definition of $\chi_u$ (see Equation~\eqref{eq Mu}) we have $\chi_u(x,u)=x^TM_uu=0$ for every $x\in V.$ So $u\in\ker M_u$ and
%  $\rank(M_u)<n$.
%\end{proof}
Denote by $C_1,\dots C_n$ the columns of the matrix $M_u$ and let $x=(x_i)_{i=1}^n$ and $u=(u_i)_{i=1}^n.$
Let  $M_u^{(i)}$ be the $(n-1)\times(n-1)$-submatrix of $M_u$ obtained by deleting its $i$-th column and  $i$-th row.
For any $x\in V,$ the condition $x^TM_uu=0$ is equivalent to
\[(x^T\cdot C_1) u_1+(x^T\cdot C_2) u_2+ \cdots +(x^T\cdot C_i) u_i+ \cdots+(x^T\cdot C_n) u_n=0,\]
where $x^T\cdot C_i:=\sum_{j=1}^n x_jc_{ji}$ with $C_i=(c_{ji})_{j=1}^n.$
Since $u_i\not=0$, we have
\[(x^T\cdot C_1) \frac{u_1}{u_i}+(x^T\cdot C_2) \frac{u_2}{u_i}+ \cdots +x^T\cdot C_i+\cdots +(x^T\cdot C_{n}) \frac{u_{n}}{u_i}=0,\]
i.e.
 \[x^T\cdot (\frac{u_1}{u_i}C_1+\cdots +C_i+\cdots+ \frac{u_{n}}{u_{i}}C_{n})=0.\]
The previous condition holds for any $x\in V,$ hence
\[C_i=-\sum_{\tiny{\begin{array}{l}
j=1\\
j\not=i
\end{array}}}^{n}\frac{u_j}{u_i}C_j.\]

  As $M_u$ is antisymmetric, the same
argument can be applied also to the $i$-th row of $M_u.$
\end{proof}

The following corollaries are straightforward.
\begin{corollary}\label{rank Mu-cor}
$\rank(M_u)\leq n-1.$
\end{corollary}
\begin{corollary}\label{corollary- H pole}
The point $[u]$ is a $H$-pole if and only if $\rank(M_u)\leq n-2.$
\end{corollary}
Note that the matrix $M_u$ is antisymmetric; hence its rank must be an even number.
By Corollary~\ref{corollary- H pole} it is clear that if $n$ is even then every point of $\PG(V)$ is a $H$-pole and this holds for any hyperplane $H$ of $\cG_3(V).$
More precisely, the degree of the point $[u]$  is $\delta(u)=(n-1)-\rank(M_u).$
So, it is straightforward to see that the set of all the $H$-poles of degree at least $t$  is either the whole of $\PG(V)$ or the determinantal variety of $\PG(V)$ described by the condition
 $\rank(M_u)\leq (n-1)-t\leq n-2$;
 see \cite[Lecture 9]{harris} for some properties of these varieties.
 Furthermore all entries of $M_u$ are linear homogeneous polynomials in the coordinates of $u$; so the condition  $\rank(M_u)\leq n-2$ provides algebraic conditions on the coordinates of $u$ for $[u]$
 to be a $H$-pole.

In Table~\ref{tab2} and Table~\ref{tab3} of Appendix~\ref{appendix} we have explicitly written the matrices $M_u$ associated to the trilinear forms $h$ of type $T_i$ of Table~\ref{Tab F} where $u=(u_i)_{i=1}^ n.$
To simplify the notation, when $\rank(h)<\dim(V)$, we have just written the $\rank(h)\times\rank(h)$-matrix associated to $h_u|_{V/\Rad(h)}.$

To determine the elements of the upper radical of $H$, namely the lines $\ell=[ x,y]$ of $\PG(V)$ with the property that any plane through them is in $H$, we need to determine  conditions on $x$ and $y$ such that the linear functional
 \begin{equation}
   \label{eq2}
   \tilde{h}_{xy}\colon V\rightarrow \KK,\,\,\tilde{h}_{xy}(u)=h(u\wedge x\wedge y)
\end{equation}
  is null.
 To do this, it is sufficient to require that $\tilde{h}_{xy}$ annihilates on the basis vectors of $V$, i.e. $\tilde{h}_{xy}(e_i)=0,$ for every $i=1,\dots, n$.

 \subsection{A determinantal variety}
 \label{ssec2}
Let $h$ be a trilinear form associated to the hyperplane $H$ and for any $u\in V$, let $\chi_u$ be the alternating bilinear form as in Equation~\eqref{eq Mu} whose representative matrix is $M_u.$
With $1\leq i\leq n$, denote by $M_u^{(i)}$ the principal submatrix of $M_u$ obtained by
deleting its $i$-th row and its $i$-th column.
The matrix $M_u^{(i)}$ is a $(n-1)\times (n-1)$-antisymmetric matrix whose entries are linear functionals defined  over $\KK$;   so,
  its determinant is a polynomial of degree $n-1$ in the unknowns $u_1,\ldots,u_n$
  which is a square in the ring $\KK[u_1,\ldots,u_n]$, that is
  there exists a polynomial  $d_i(u_1,\ldots,u_n)$ with $\deg d_i(u_1,\ldots,u_n)=(n-1)/2$ such that
    \begin{equation}\label{edi}
     \det M_u^{(i)}=(d_i(u_1,\ldots,u_n))^2.
     \end{equation}
Define $g_i(u_1,\ldots,u_n)$ to be the polynomial in $\KK[u_1,\ldots,u_n]$ such that
$$d_i(u_1,\ldots,u_n)=u_i^{\alpha_i}g_i(u_1,\ldots,u_n)$$ where $\alpha\in{\mathbb N}$ and
$u_i^{\alpha_i+1}$ does not divide $d_i(u_1,\ldots,u_n)$.

\begin{theorem}
\label{lp}
The set $P(H)$ of $H$-poles is either the whole pointset of $\PG(V)$ or there exists an index $i$, $1\leq i\leq n$,  such that $P(H)$
is an algebraic hypersurface of $\PG(V)$ with equation $g_i(u_1,\ldots,u_n)=0$.
\end{theorem}
\begin{proof}
  Suppose $P(H)\neq\PG(V)$.
  By \cite{DS10}, $P(H)$ is an algebraic hypersurface admitting an equation of degree $(n-3)/2$; so
  $P(H)$ contains at most $(n-3)/2$ coordinate hyperplanes of the form $\Pi_j: u_j=0$.
  Then, there exists $i$ with $1\leq i\leq n$
   such that
   $\Pi_i\not\subseteq P(H)$.
   By Corollary~\ref{corollary- H pole}, $[u]\in \PG(V)$ is a $H$-pole if and only if $\rank(M_u)\leq n-2.$
   If $n$ is even, this always happens. So, assume $n$ odd.
   We shall work over the algebraic closure $\overline{\KK}$ of $\KK$.
  Take $u=(u_1,\ldots,u_n)$ with $u_j\in\overline{\KK}$. We can
  regard $u$ as
  a vector in coordinates
  with respect to the basis induced by $E$ on
  $\overline{V}:=V\otimes\overline{\KK}$.
  In any case, the matrix $M_u$ is antisymmetric and its
  entries are  homogeneous linear functionals in
  $u_1,\ldots,u_n$ defined  over the field $\KK$.

  Consider the points $[u]$ with $u_i\not=0.$
  By Proposition~\ref{rank Mu-prop},
  the $i$-th row/column of $M_u$ is a linear combination
  of the remaining $n-1$ rows/columns. So, $\rank M_u=\rank M_u^{(i)}$.

  Let $d_i(u_1,\ldots,u_n)$ be as in Equation~\eqref{edi}. We now show that
  $$d_i(u_1,\ldots,u_{i-1},0,u_{i+1},\ldots,u_n)=0$$ for all $u_1,\ldots,u_{i-1},u_{i+1},\ldots,u_n\in\overline{\KK}$.
  Indeed, when
  $u_i=0$, by Proposition~\ref{rank Mu-prop},
  there exists a column $C_j$ of $M_u$, $u_j\not=0$, such that
\[C_j=-\sum_{\tiny{\begin{array}{l}
k=1\\
k\not=j
\end{array}}}^{n}\frac{u_k}{u_j}C_k.\]
So, the $j$-th column of $M_u^{(i)}$ is also a linear combination of the other columns of $M_u^{(i)}.$ Hence
 $\det M_u^{(i)}=0=(d_i(u_1,\ldots,0,\ldots,u_{n}))^2$.

  Since $\overline{\KK}$ is algebraically closed, we have
  \[ d_i(u_1,\ldots,u_n)= u_i d'_i(u_1,\ldots,u_n), \]
  with $d'_i(u_1,\ldots,u_n)$ a polynomial in $\KK[u_1,\ldots,u_n]$
  with $\deg d'_i=(n-3)/2$. We remark that the unknowns
  of the polynomials may assume their values in $\overline{\KK}$ but the  coefficients are
  all in $\KK.$
  By Corollary~\ref{corollary- H pole}, a point $[u]\in\PG(\overline{V})\setminus\Pi_i$ (i.e. $u_i\neq0$)
  is a $H$-pole if and only if  $\det M_u^{(i)}=0$, i.e. $d'_i(u_1,\ldots,u_n)=0.$

  Denote by $\overline{\Gamma}_i$ the algebraic variety
  (over $\overline{\KK}$) of equation $d_i'(u_1,\ldots,u_n)=0$.
  Since we are assuming $\Pi_i\not\subseteq P(H)$ and
   $P(H)\setminus\Pi_i=\overline{\Gamma}_i\setminus\Pi_i$, we have
  \[ P(H)=C(P(H)\setminus\Pi_i)=C(\overline{\Gamma}_i\setminus\Pi_i), \]
  where $C(X)$ denotes the projective closure of $X$ in $\PG(\overline{V})$ with $\Pi_i$
   regarded as the hyperplane at infinity.
   Note that since $P(H)$ is an algebraic variety of $\PG(\overline{V})$ which does
   not contain $\Pi_i$, the points
   of $P(H)\cap\Pi_i$ are exactly those of the projective closure $C(P(H)\setminus\Pi_i)\cap\Pi_i$.
   The same applies to the points at infinity of the affine variety $\overline{\Gamma}_i\setminus\Pi_i$.

  Suppose $u_i^{\beta}|d_i'(u_1,\ldots,u_n)$ and $u_i^{\beta+1}\not| d_i'(u_1,\ldots,u_n)$ for
  $\beta\in{\mathbb N}$.

  Then,
  $d_i'(u_1,\ldots,u_n)=u_i^{\beta}g_i(u_1,\ldots,u_n)$ and
  $g_i(u_1,\ldots,u_n)=0$ is an equation for $C(\overline{\Gamma}_i\setminus\Pi_i)=P(H)$.
  This completes the proof.
\end{proof}
Observe that if $\alpha_i=1$, then $g_i(u_1,\ldots,u_n)=0$ is exactly an equation of degree $(n-3)/2$ for
$P(H)$.

If $n$ is even then each point of $\PG(V)$ is a $H$-pole. Likewise, when $n$ is odd and $H$ is defined by a trilinear form $\chi$ of rank less than $n$, then also each point
of $\PG(V)$ is a $H$-pole. Indeed, whenever $R_{\downarrow}(H)=\Rad(\chi)$ is
non-trivial,  for any $[u]\in \PG(V)$ there exists a line $\ell$ through $[u]$ in $R^{\uparrow}(H)$ which
meets $\Rad(\chi)$ and so $[u]$ is a $H$-pole.
In any case the above conditions, while sufficient, are not in general necessary;
in fact, in Section~\ref{decomposable}
we shall provide a construction which might lead to alternating forms of rank $n$
with $n$ odd and $R_{\downarrow}(H)=\emptyset$ such that all points of $\PG(V)$ are $H$-poles.

Note that
Theorem~\ref{lp} shows that the variety of the $H$-poles
in $\PG(V)$ admits at least one equation over $\KK$ of degree at most
$(n-3)/2$; this does not mean that $(n-3)/2$ is the minimum degree
for such an equation or that the polynomial $g_i(u_1,\ldots,u_n)$ generates
the radical ideal of such variety (even over $\KK$).
For instance, in the case of symplectic hyperplanes,
see Section~\ref{symp hyp sec}, the variety of poles is always a
hyperplane of $\PG(V)$ and so it admits an equation of degree $1$ for any odd $n$.

\subsection{A family of reducible hyperplanes}
\label{rehy}
It is an interesting problem to investigate which algebraic varieties
might arise as set of $H$-poles; as noted before such varieties will always be
skew-symmetric determinantal varieties \cite{Ha}.
We leave the development of this study to a further paper. However,
with this aim in mind, we
give here a construction for hyperplanes whose variety of poles are reducible.
\begin{theorem}
  \label{c1}
  Suppose $V=V_1+V_2$ with $(e_1,\ldots,e_{n_1})$ and $(e_{n_1},\ldots,e_n)$ bases of
  $V_1$ and $V_2$ respectively.  For $i=1,2$ let $H_i$  be a hyperplane of $\cG_3(V_i)$ whose $H_i$-poles  in $\PG(V_i)$
  satisfy respectively the equations $f_1(u_1,\ldots,u_{n_1})=0$ and $f_2(u_{n_1},\ldots,u_n)=0$.
 Then there exists a hyperplane $H$ of $\cG_3(V)$ whose set of $H$-poles defines a variety of $\PG(V)$ with equation
  $f(u_1,\ldots,u_n)=0$ where
  \[ f(u_1,\ldots,u_n):=u_{n_1}\cdot f_1(u_1,\ldots,u_{n_1})f_2(u_{n_1},\ldots,u_{n}). \]
\end{theorem}
\begin{proof}
 For $i=1,2$ denote by  $\overline{h}_i$ the trilinear form on $V_i$  defining $H_i$ and consider the extension $h_i\colon V\times V\times V\rightarrow \KK$ given by
\[h_i(x,y,z)=\overline{h}_i(\pi_i(x),\pi_i(y),\pi_i(z))\]
where $\pi_1\colon V\rightarrow V_1$ is the projection on $V_1$ along $\langle e_{n_1+1},\dots e_n\rangle$ and
$\pi_2\colon V\rightarrow V_2$ is the projection on $V_2$ along $\langle e_{1},\dots e_{n_1-1}\rangle.$

Put $h:=h_1+h_2$ and let $H$ be the hyperplane of $\cG_3(V)$ defined by $h.$
For any $[u]\in \PG(V)$, let $\chi_u$ as in Equation~\eqref{eq Mu} be the bilinear alternating form induced by $h$ and represented by the matrix $M_u$ with respect to the basis $(e_1,\ldots,e_{n_1},\ldots,e_n)$ of $V.$ By construction, the matrix $M_u$ has the following structure:
\[M_u=\begin{pmatrix}
M_1 &(v^1_{n_1}) & 0\\
-{(v^1_{n_1})}^T &0& -{(v^2_{n_1})}^T\\
0&(v^2_{n_1}) & M_2\\
\end{pmatrix}\]
where $\bar{M}_1:=\begin{pmatrix}
M_1 &(v^1_{n_1})  \\
-{(v^1_{n_1})}^T &0
\end{pmatrix}$ is the $n_1\times n_1$-matrix representing $\overline{h}_1$ and $\bar{M}_2=\begin{pmatrix}
0& -{(v^2_{n_1})}^T\\
(v^2_{n_1}) & M_2\\
\end{pmatrix}$ is the $(n-n_1+1)\times (n-n_1+1)$-matrix representing $\overline{h}_2.$
Note that $(v^1_{n_1})$ and $(v^2_{n_1})$ are suitable columns with entries respectively in
the rings $\KK[u_1,\ldots,u_{n_1}]$ and $\KK[u_{n_1},\ldots,u_{n}]$.

Suppose $u_{n_1}=0$. Since the entries of $\bar{M}_1$ and $\bar{M}_2$ are respectively
linear functionals in $u_1,\ldots,u_{n_1}$ and $u_{n_1},\ldots,u_{n}$ if
all entries $u_i$ of $u$ with $i\geq n_1$ are null, then $\bar{M}_2$ is a zero matrix
and $\rank M_u\leq n-2$. Likewise if all entries $u_i$ of $u$ with $i\leq n_1$ are
zero, then $\bar{M}_1$ is the zero matrix and $\rank M_u\leq n-2$.
Assume now that $u_{n_1}=0$ and that there exist $i,j$ with $i<n_1$ and $j>n_1$ such that
$u_i\neq0\neq u_j$. Then,
by Proposition~\ref{rank Mu-prop}, the first $(n_1-1)$ columns of $\bar{M}_1$ and the last $n-n_1$ columns of
$\bar{M}_2$ are a linearly dependent set; in particular, $\rank M_u\leq n-2$.
So we have proved that the hyperplane $\Pi_{n_1}:u_{n_1}=0$ is
contained in the variety of poles $P(H)$.

By construction, $\rank \bar{M}_1\leq n_1-2$ if and only if $f_1(u_1,\ldots,u_{n_1})=0$ and
$\rank \bar{M}_2\leq n-n_1-1$ if and only if $f_2(u_{n_1},\ldots,u_n)=0$.
Let $\Delta$ be the variety of equation $f_1(u_1,\ldots,u_{n_1})f_2(u_{n_1},\ldots,u_n)=0$.
Observe that $\Delta\setminus\Pi_{n_1}\subseteq P(H)\setminus\Pi_{n_1}$, as
$[u]\in\Delta\setminus\Pi_{n_1}$ implies
that either $\rank\bar{M}_1^{(n_1)}\leq\rank\bar{M}_1\leq n_1-2$
or $\rank\bar{M}_2^{(n_1)}\leq\rank\bar{M}_2\leq n-n_1-1$ and, by Proposition~\ref{rank Mu-prop} ($u_{n_1}\not=0$),
the $n_1$-th column of $M$ is a linear combinations of the columns of $\bar{M}_1^{(n_1)}$
as well as a linear combination of the columns of $\bar{M}_2^{(n_1)}$ (so it does not contribute to the rank); so $\rank M_u\leq n-2$
and $[u]\in P(H)\setminus\Pi_{n_1}$.
Conversely, suppose $[u]\in P(H)$ with $u_{n_1}\neq0$. Then, by Theorem~\ref{lp},
\[ 0=\det M_{u}^{(n_1)}=\det M_1\cdot\det M_2=\det \bar{M}_1^{(n_1)}\cdot\det \bar{M}_2^{(n_1)}. \]
If $[u]\in\PG(V)\setminus\Pi_{n_1}$, again by Theorem~\ref{lp} applied to $V_1$ and $V_2$ we have
$\det\bar{M}_1^{(n_1)}=0$ if and only if $f_1(u_1,\ldots,u_{n_1})=0$ and
$\det\bar{M}_2^{(n_1)}=0$ if and only if $f_2(u_{n_1},\ldots,u_{n})=0$.
So
$P(H)\setminus\Pi_{{n_1}}\subseteq\Delta\setminus\Pi_{n_1}$, whence
\[ P(H)\setminus\Pi_{{n_1}}=\Delta\setminus\Pi_{n_1} \]
Since $\Pi_{n_1}\subseteq P(H)$, we have
\[ P(H)=(P(H)\setminus\Pi_{n_1})\cup\Pi_{n_1}=(\Delta\setminus\Pi_{n_1})\cup\Pi_{n_1} \]
and, consequently $f(u_1,\ldots,u_n)=u_{n_1}\cdot f_1(u_1,\ldots,u_{n_1})\cdot f_2(u_{n_1},\ldots,u_n)$
is an equation for $P(H)$.
\end{proof}
\begin{corollary}
\label{cch}
  Let $V$ be a vector space of odd dimension $n\geq 5$ over a field $\KK$.
  Then there exists a  hyperplane $H$ of $\cG_3(V)$  whose set of $H$-poles is the
  union of $(n-3)/2$ distinct hyperplanes of $\PG(V)$.
\end{corollary}
\begin{proof}
We proceed by induction on $n$ odd.
  If  $n=5$, consider the hyperplane of $\cG_3(V)$  defined by the trilinear form  $h:=\underline{123}+\underline{345}.$ It is easy to verify that its poles are all the points of the hyperplane of equation $u_3=0.$

By induction hypothesis, suppose that the thesis holds for vector spaces of odd dimension $n$. We shall prove it also holds for vector spaces of (odd) dimension $n+2.$
Let $V$ with $\dim(V)=n+2$ and let $(e_i)_{i=1,\dots, n+2}$ be a given basis of $V.$
Put $V_1:=\langle e_i\rangle_{1\leq i\leq n}$ and $V_2:=\langle e_{n},e_{n+1},e_{n+2}\rangle.$ Clearly $V=V_1+V_2.$
As $\dim(V_1)=n$ we can apply the induction hypothesis. So, there exists a trilinear form $\overline{h}_1$ on $V_1$ (defining a hyperplane $H_1$ of $\cG_3(V_1)$) such that its set of poles is the union of $(n-3)/2$ distinct hyperplanes of $\PG(V)$. Equivalently, we can assume without loss of generality that any $H_1$-pole satisfies the equation $g_1(u_1,\ldots,u_n)=0$ where
$g_1(u_1,\ldots,u_n):=\prod_{i=1}^{(n-3)/2} u_{2i+1}.$

Let $\overline{h}_2$ be the trilinear form on $V_2$ defined by $\overline{h}_2=\underline{(n)(n+1)(n+2)}$. Clearly
$\overline{h}_2$ has no pole in $\PG(V_2).$

By (the proof of) Theorem~\ref{c1}, we can consider the hyperplane $H$ of $\cG_3(V)$ defined by the sum of the extensions $h_1$ and $h_2$ to $V$ of $\overline{h}_1$ and $\overline{h}_2$ (see the beginning of the proof of Theorem~\ref{c1} for the definition of $\overline{h}_1$ and $\overline{h}_2$). Then, the set of $H$-poles is a variety of $\PG(V)$ with equation $g(u_1,\ldots,u_{n+2})=0$ where
\[ g(u_1,\ldots,u_{n+2}):=g_1(u_1,\ldots,u_{n})\cdot u_n=\left(\prod_{i=1}^{(n-3)/2} u_{2i+1}\right) \cdot u_n=
\prod_{i=1}^{(n-1)/2}u_{2i+1}.\]\end{proof}
%We shall call the hyperplanes $H$ arising in Corollary~\ref{cch} \emph{fully reducible}.

% \begin{prop}
%   Let $H$ be a fully reducible hyperplane obtained as in Corollary~\ref{cch} and let $\Pi_i$ with
%   $i=1,\ldots,(n-3)/2$ be the hyperplanes of $\PG(V)$ comprising $P(H)$.
%   Then, on each hyperplane $\Pi_i$ is defined a non-degenerate symplectic form $\beta_i(x_i,y_i)$
% \end{prop}

\section{Constructions of families of hyperplanes}\label{sec 3}
In this section we will explain some general constructions yielding large families of hyperplanes of $k$-Grassmannians. More precisely, in Sections~\ref{extensions} and \ref{expansions} we shall briefly recall (without
proofs) two constructions already introduced in~\cite{ILP17} while in Section~\ref{decomposable} we will present a new one.

We first need to give the following definition which extends the definition of $\cS_p(H)$ given in Section~\ref{geom of poles}. For a $(k-2)$-subspace $X$ of $V$, let $(X)G_k$ be the set of $k$-subspaces of $V$ containing $X$. This is a subspace of $\cG_k(V)$. Let $(X)\cG_k$ be the geometry induced by $\cG_k(V)$ on $(X)G_k$ and put $(X)H := (X)G_k\cap H$. Then $(X)\cG_k \cong \cG_2(V/X)$ and either $(X)H = (X)G_k$ or $(X)H$ is a hyperplane of $(X)\cG_k$. In either case, the point-line geometry ${\cS}_X(H) = ((X)G_{k-1}, (X)H)$ is a polar space of symplectic type (possibly a trivial one, when $(X)H = (X)G_k)$). Let $R_X(H) := \Rad({\cS}_X(H))$ be the radical of ${\cS}_X(H)$.

\subsection{Extensions and trivial extensions}\label{extensions}
Let $V = V_0\oplus V_1$ be a decomposition of $V$ as the direct sum of two non-trivial subspaces $V_0$ and $V_1$. Put $n_0 := \dim(V_0)$ and assume that $n_0 \geq k$ ($\geq 3$). Let $\chi_0:V_0\times\cdots\times V_0\rightarrow \KK$ be a non-trivial $k$-linear alternating form on $V_0$. The form $\chi_0$ can naturally be extended to a $k$-linear alternating form $\chi$ of $V$ by setting
\begin{equation}\label{extended form}
\left.\begin{array}{rcll}
\chi(x_1,\dots, x_k) & = & 0 & \mbox{if} ~ x_i\in V_1 ~ \mbox{for some} ~ 1\leq i\leq k,\\
\chi(x_1,\dots, x_k) & = & \chi_0(x_1,\dots, x_k) & \mbox{if} ~ x_i\in V_0 ~\mbox{for all} ~ 1\leq i\leq k,
\end{array}\right\}
\end{equation}
and then extending by (multi)linearity. Let $H_\chi$ be the hyperplane of $\cG_k(V)$ defined by $\chi$.
Then, the following properties hold:
\begin{theorem}[\cite{ILP17}]\label{extended hyperplane}
  Let $\chi_0$ be a $k$-alternating linear form on $V_0$ and $n_0=\dim V_0$; define $\chi$ as
  in~\eqref{extended form}.
  For $n_0=k$, put $H_0=\emptyset$; otherwise, let $H_0$ be the hyperplane of
  $\cG_k(V_0)$ defined by $\chi_0$.
  Let also $\pi: V\rightarrow V_0$ be the projection of $V$ onto $V_0$ along $V_1$. Then,
  \begin{enumerate}[(1)]
  \item
    $H_\chi ~ = ~ \{X\in \cG_k(V)~\colon~ \mbox{either}~ X\cap V_1 \neq 0 ~\mbox{or}~ \pi(X)\in H_0\}$.
  \item\label{ext-1} $R_\downarrow(H_\chi) ~ = ~ \langle R_\downarrow(H_0) \cup [V_1]\rangle$ ~where the span in taken in $\PG(V)$.
\item\label{ext-2} $R^\uparrow(H_\chi) ~= ~ \{X\in G_{k-1}(V)~ \colon~  \mbox{either}~ X\cap V_1\neq 0 ~\mbox{or}~ \pi(X)\in R^\uparrow(H_0)\}.$
\item  When $k = 3$, the points $[p]\not\in [V_1]$ have degree $\delta(p) = \delta_0(\pi(p)) + n-n_0$, where $\delta_0(\pi(p))$ is the degree of $[\pi(p)]$ with respect to $H_0$. The points $p\in [V_1]$ have degree $n-1$.
\end{enumerate}
\end{theorem}

We call $H_\chi$ the \emph{trivial extension of $H_0$  centered at $V_1$}
(also \emph{extension of $H_0$ by $V_1$}, for short) and we denote it by the symbol $\ext_{V_1}(H_0).$ % $H_0\odot V_1$.
When $k=n_0$ we have $H_0=\emptyset$;  we shall
call $\ext_{V_1}(\emptyset)$ the \emph{trivial hyperplane centered at} $V_1$.
In this case Theorem~\ref{extended hyperplane}
can also be rephrased as follows, with no direct mention of $H_0$.

\begin{prop}[\cite{ILP17}]\label{ext hyp sing}
  Let $H=\ext_{V_1}(\emptyset)$ be a trivial hyperplane of $\cG_k(V)$.
  Then
  \[H ~ = ~ \{X\in G_k(V)~\colon~ X\cap V_1 \neq 0\}.\]
  Moreover, $R_\downarrow(H) ~ = ~  [V_1]$, $R^\uparrow(H) ~= ~ \{X\in \cG_{k-1}(V)~ \colon~ X\cap V_1\neq 0\}$ and, for $X\in \cG_{k-2}(V)$, if $X\cap V_1\neq 0$ then $R_X(H) = [V/X]$, otherwise $R_X(H) = [(V_1+ X)/X]$.
\end{prop}

By construction, the lower radical of a trivial extension is never empty. The following theorem
shows that the converse is also true, namely  if  $R_\downarrow(H)\neq\emptyset$ then $H$ is a trivial extension, possibly a trivial hyperplane.

If $S$ is a subspace of $V$ with $\dim(S)>k$, denote by $H(S):=\cG_k(S)\cap H$ the hyperplane of $\cG_k(S)$ induced by $H.$

\begin{theorem}[\cite{ILP17}]\label{ext hyp th}
  Suppose $R_\downarrow(H) \neq \emptyset$ and let $S,S'$ be complements in $V$ of the subspace $R < V$ such that $[R] = R_\downarrow(H)$. Then
  \begin{enumerate}[(1)]
  \item
 %   $H ~=~ H(S)\odot R_\downarrow(H)$;
     $H ~=~ \ext_R(H(S))$;
  \item
    $H(S)\cong H(S')$;
  \item $R_\downarrow(H(S)) = \emptyset$.
  \end{enumerate}
\end{theorem}

Each hyperplane $H$ of $\cG_k(V)$ defined by a $k$-linear alternating form $h$ with $\rank(h)<\dim V$ is
clearly a trivial extension of a hyperplane $H'$ of $\cG_k(V')$ with $\dim V'=\rank(h)$, since $V=\Rad(h)\oplus V'.$

\subsection{Expansions}\label{expansions}

Let $V_0$ be a hyperplane of $V$ and $H_0$ a given hyperplane of $\cG_{k-1}(V_0)$. Assume $k \geq 3$; hence $V$ has dimension $n \geq 4$. Put
\[\expa(H_0) ~:= ~ \{X\in \cG_k(V) ~\colon~ \mbox{either}~ X\subset V_0 ~\mbox{or}~X\cap V_0 \in H_0\}.\]

\begin{theorem}[\cite{ILP17}]\label{sympl hyperplane}
The set $\expa(H_0)$ is a hyperplane of $\cG_k(V)$. Moreover,
\begin{enumerate}[(1)]
\item\label{sym-1}
 $R_\downarrow(\expa(H_0)) ~= ~ R_\downarrow(H_0)$.
\item\label{sym-2}
 $R^\uparrow(\expa(H_0)) ~= ~ H_0\cup\{X\in \cG_{k-1}(V)\setminus \cG_{k-1}(V_0) ~\colon~ X\cap V_0 \in R^\uparrow(H_0)\}$.
\item\label{sym-3}
 For $X\in \cG_{k-2}(V)$, if $X\subseteq V_0$ with $X\in R^\uparrow(H_0)$ then $R_X(\expa(H_0)) = {\cS}_X(\expa(H_0)) = (X)G_{k-1}$ (the latter being computed in $V$). If $X\subseteq V_0$ but $X\not\in R^\uparrow(H_0)$ then $R_X(\expa(H_0)) = (X)H_0$ (a subspace of $\PG(V_0/X)$). Finally, if $X\not\subseteq V_0$, then $R_X(\expa(H_0)) ~=~ \{\langle x, Y\rangle~\colon~ Y\in R_{X\cap V_0}(H_0)\}$ for a given $x\in X\setminus V_0$, no matter which.
\end{enumerate}
\end{theorem}

We call $\expa(H_0)$ the \emph{expansion} of $H_0$. A form $h:\bigwedge^kV\rightarrow \KK$ associated to $\expa(H_0)$ can be constructed as follows.
Suppose $h_0:\bigwedge^{k-1}V_0\rightarrow \KK$ is the $(k-1)$-alternating linear form defining $H_0$.
Suppose $V_0=\langle e_1,\ldots, e_{n-1}\rangle$ where $E=(e_i)_{i=1}^n$ is the given basis of $V$. Recall that $\{e_{i_1}\wedge\cdots\wedge e_{i_k} ~\colon ~ 1\leq i_1 < \ldots < i_k \leq n\}$ is a basis of $\bigwedge^kV$. Put
\begin{equation}\label{expanded form}
h(e_{i_1}\wedge\cdots\wedge e_{i_k}) = \left\{\begin{array}{l}
0 ~\mbox{if}~ i_k < n,\\
h_0(e_{i_1}\wedge\cdots\wedge e_{i_{k-1}}) ~\mbox{if}~ i_k = n.
\end{array}\right.
\end{equation}
and extend it by linearity. It is easy to check that the form $h$ defines $\expa(H_0)$.

We now recall some properties linking expansions and trivial extensions which might be of
use in investigating the geometries involved.
\begin{theorem}[\cite{ILP17}]  \label{symp hyp}
  Let $H_0$ be a hyperplane of $\cG_k(V_0)$; then
  \begin{enumerate}
  \item $R_\downarrow(\expa(H_0)) = \emptyset$ if and only if $R_\downarrow(H_0) = \emptyset$;
  \item denote by $S_0$ a complement of $R_0\leq V$ such that $[R_0]=R_{\downarrow}(H_0)$;
    then,
    $\expa(H_0)=\ext_{R_0}(\expa(H_0(S_0)))$ where $H_0(S_0)$ is the hyperplane induced on $S_0$ by $H_0$;
  \item if $H_0$ is trivial, then $\expa(H_0)$ is also trivial with center $R_{\downarrow}(H_0)$.
  \end{enumerate}
\end{theorem}

\subsubsection{Symplectic hyperplanes}\label{symp hyp sec}
Assume now $k = 3$ and take $H_0$ to be a hyperplane of $\cG_2(V_0)$ (hence defined by a bilinear alternating form of $V_0$). The point-line geometry ${\cS}(H_0) = (G_1(V_0), H_0)$ having as points all points of $[V_0]$ and as lines all elements in $H_0$, is a polar space of symplectic type. The upper and lower radical of $H_0$ are mutually equal and coincide with the radical $R(H_0)$ of ${\cS}(H_0)$.

First suppose  that ${\cS}(H_0)$ is non-degenerate. Then $n-1$ is even, whence $n \geq 5$. Claims (\ref{sym-1}) and (\ref{sym-2}) of Theorem~\ref{sympl hyperplane} imply that
$R_\downarrow(\expa(H_0)) = \emptyset$ and $R^\uparrow(\expa(H_0)) = H_0$; thus the geometry of poles $\cP(\expa(H_0))$ of $\expa(H_0)$ coincides  precisely with the symplectic polar space ${\cS}(H_0).$ %and
%${\cP}(\expa(H_0)) = {\cS}(H_0)$.
In particular, the points of $[V]\setminus[V_0]$ are smooth while those of $[V_0]$ are poles of degree $1$.
Motivated by the above remark we call $\expa(H_0)$ a \emph{symplectic hyperplane} whenever
$R(H_0) = \emptyset.$

Assume now that ${\cS}(H_0)$ is degenerate, i.e. $\Rad({\cS}(H_0)) \neq 0$.
In this case, $R_\downarrow(\expa(H_0))\neq 0$ since $R_\downarrow(\expa(H_0))=\Rad({\cS}(H_0))$; so $\expa(H_0)$ is either a trivial extension
of a symplectic hyperplane by $\Rad({\cS}(H_0))$ (this happens when $\dim(\Rad({\cS}(H_0)) < n-3$) or a trivial hyperplane centered at $\Rad({\cS}(H_0))$ (this happens when $\dim(\Rad({\cS}(H_0)) =  n-3$).

\subsection{Block decomposable hyperplanes}
\label{decomposable}
The construction of block decomposable hyperplanes can be done for general $k\geq 3$ but we will give the details for the case $k=3.$

Suppose $V=V_0\oplus V_1$. Any vector $x\in V$ can then be uniquely written as $x=x_0+x_1$ with $x_0\in V_0$ and $x_1\in V_1.$ For $i=0,1$ let $\bar{h}_i:\bigwedge^3V_i\to\KK$ be a linear functional defining the hyperplane $H_i$ of $\cG_3(V_i)$. Consider the extension $h_i:\bigwedge^3V\rightarrow \KK$ of $\bar{h}_i$ to $V$ given by
\[ h_i(x\wedge y\wedge z)=\bar{h}_i(x_i\wedge y_i\wedge z_i) \]
where  $x=x_0+x_1,y=y_0+y_1,z=z_0+z_1\in V$ and  $x_i,y_i,z_i\in V_i.$

Let $h:=h_0+h_1$ be the trilinear form of $V$ defined by the sum of $h_0$ and $h_1.$ So,
\[
  h((x_0+x_1)\wedge(y_0+y_1)\wedge(z_0+z_1))=\bar{h}_0(x_0\wedge y_0\wedge z_0)+
  \bar{h}_1(x_1\wedge y_1\wedge z_1). \]

Then the hyperplane of $\cG_3(V)$ defined by $h$ is called a \emph{block decomposable hyperplane}
arising from $H_0$ and $H_1$ and it will be denoted by $\dec(H_0,H_1).$

\begin{theorem}
  \label{t3}
 Let $H:=\dec(H_0,H_1)$ be a block decomposable hyperplane of $\cG_3(V).    $  % with $V=V_0\oplus V_1$.
  %Denote by $H_i$ the hyperplanes induced by $h_i$ on $\bigwedge^3V_i$, $i=1,2$ and
  %let $h=h_0+h_1$ and $H:=H_h$.
  Then the following hold:
  \begin{enumerate}
 % \item $H_0\oplus H_1\leq H$;
  \item The poles of $H$ are all the points of $\PG(V_0\oplus V_1)$;
  \item\label{c2} $R^{\uparrow}(H)=\{ \ell\in\cG_2(V) : (\pi_0(\ell)\in R^{\uparrow}(H_0) \mbox{ or }
    \dim(\pi_0(\ell))<2) \mbox{ and } (\pi_1(\ell)\in R^{\uparrow}(H_1) \mbox{ or }
    \dim(\pi_1(\ell))<2)\}$
 %   \pi_0(r)\in R^{\uparrow}(H_0) \mbox{ or }
 %   \pi_1(r)\in R^{\uparrow}(H_1) \mbox{ or } r\cap V_0\neq\emptyset\neq r\cap V_1 \}$,
    where $\pi_i\colon V\rightarrow V_i$ is the projection of $V$
    onto $V_i$ along $V_j$ ($j\neq i$, $i=0,1$).
    Denote by $\varepsilon_2\colon \cG_2(V)\rightarrow \PG(\bigwedge^2 V)$ the Pl\"ucker embedding of the $2$-Grassmannian $\cG_2$.
    We have
    \[\varepsilon_2(R^{\uparrow}(H))=[\varepsilon_2(R^{\uparrow}(H_0))+
    \varepsilon_2(R^{\uparrow}(H_1))+V_0\wedge V_1]\cap\varepsilon_2(\cG_2)\]
    where
    $V_0\wedge V_1:=\langle v_0\wedge v_1: v_0\in V_0, v_1\in V_1\rangle.$
 \end{enumerate}
\end{theorem}

\begin{proof}
  Put $n_0=\dim V_0$ and $n_1=\dim V_1$. Let $u\in V$ where $u=u_0+u_1\in V$, $u_0\in V_0$ and $u_1\in V_1.$ Denote by
  $M_u$ the matrix of the bilinear form $\chi_u(x,y):=h(u\wedge x\wedge y)$.
  Then, $M_u$ is a block matrix of the form
  \[ M_u=\begin{pmatrix}
      M_{u_0} & 0 \\
      0 & M_{u_1}
    \end{pmatrix} \]
where $M_{u_i}$ is the matrix representing the form $\chi_u^i(x,y):=\bar{h}_i(u\wedge x\wedge y)$ associated to the hyperplane $H_i$ of $\cG_3(V_i).$

  % The point $[u]$ is a pole if, and only if, $\rank M_u\leq n-2$.
  % On the other hand $u_i\in\ker M_{u_i}$ for $i=1,2$; so
  % \[ \rank (M_u)\leq\rank(M_{u_0})+\rank(M_{u_1})\leq (n_0-1)+(n_1-1)=n-2. \]
  % It follows that $[u]$ is a pole.
  \par
    For any $x,y,u\in V$ with $x=x_0+x_1$, $y=y_0+y_1$, $u=u_0+u_1$ and $x_i,y_i,u_i\in V_i$, we have by definition of decomposable hyperplane,
  \[ \chi_u(x,y)=x^TM_uy=x_0^TM_{u_0}y_0+
    x_1^TM_{u_1}y_1. \]
   By Corollary~\ref{rank Mu-cor}, $\rank(M_{u_i})\leq n_i-1$.
%  Observe that $M_{u_0}u_0=0$ and $M_{u_1}u_1=0$; in particular, if
%  $\dim V_i=n_i$, $\dim V=n_0+n_1=n$ we have that $\rank (M_{u_0})\leq n_0-1$
%  and $\rank (M_{u_1})\leq n_1-1$.
So, $\rank M_{u}\leq (n_0-1)+(n_1-1)=n-2$.
 By Corollary~\ref{corollary- H pole}, $[u]=[u_0+u_1]$ is a pole.
  \par
  A line $\ell=\langle x,y\rangle$ is in the upper radical $R^{\uparrow}(H)$ if,
  and only if, for any choice of $u\in V$ we have $\chi_u(x,y)=0$.
  Since
  \[ \chi_u(x,y)=\chi_{u_0}(x_0,y_0)+\chi_{u_1}(x_1,y_1), \]
  where $x_i,y_i\in V_i$ and $x=x_0+x_1$, $y=y_0+y_1$,
  we have $\ell\in R^{\uparrow}(H)$ if and only if for all $u_i\in V_i$ and $i=0,1$,
  $\chi_{u_i}(x_i,y_i)=0.$
  This holds if and only if $(\pi_0(\ell)\in R^{\uparrow}(H_0) \mbox{ or }
    \dim(\pi_0(\ell))<2) \mbox{ and } (\pi_1(\ell)\in R^{\uparrow}(H_1) \mbox{ or }
    \dim(\pi_1(\ell))<2).$
    The first part of claim~\ref{c2} is proved.

  Consider the following subspace of $\PG(\bigwedge^3 V)$
  \[ \cL:=[\varepsilon_2(R^{\uparrow}(H_0))+\varepsilon_2(R^{\uparrow}(H_1))+V_0\wedge V_1] \]
  where  $V_0\wedge V_1:=\langle v_0\wedge v_1: v_0\in V_0, v_1\in V_1\rangle.$

We claim that any point in  $\varepsilon_2(R^{\uparrow}(H_0))+\varepsilon_2(R^{\uparrow}(H_1))+V_0\wedge V_1 $ is in $\varepsilon_2(R^{\uparrow}(H)).$

\noindent Since $R^{\uparrow}(H_i)\subseteq R^{\uparrow}(H)$ (for $i=0,1$) and $\varepsilon_2^{-1}((V_0\wedge V_1)\cap \varepsilon_2(\cG_2(V)))\subseteq R^{\uparrow}(H)$, by the first part of claim~\ref{c2}, %  ,R^{\uparrow}(H_1)$ and the preimage of the set $(V_0\wedge V_1)\cap\varepsilon_2(\cG_1(V))$
%  are all contained in $R^{\uparrow}(H)$,
the inclusion
  $\cL\subseteq\varepsilon_2(R^{\uparrow}(H))$ is immediate.

 Suppose now $\langle (x_0+x_1)\wedge (y_0+y_1)\rangle\in\varepsilon_2(R^{\uparrow}(H))$.
%Then $(x_0\wedge y_0)\in \varepsilon_2(R^{\uparrow}(H_0))$ and
  %$(x_1\wedge y_1)\in\varepsilon_2(R^{\uparrow}(H_1))$; so
We can write (by the first part of claim~\ref{c2})
  \[ (x_0+x_1)\wedge(y_0+y_1)=\underbrace{(x_0\wedge y_0)}_{\in\varepsilon_2(R^{\uparrow}(H_0))}+
      \underbrace{(x_0\wedge y_1)+(x_1\wedge y_0)}_{\in V_0\wedge V_1}+\underbrace{(x_1\wedge y_1)}_{\in \varepsilon_2(R^{\uparrow}(H_1))}\in\cL. \]
    The thesis follows.
\end{proof}

 We remark that, with some slight abuse of notation,
the extension $\ext_{V_1}(H_0)$ of an hyperplane $H_0$ can always
be regarded as a special case of a block decomposable hyperplane, where the
form $\bar{h}_1$ defined over $V_1$ is identically null.

The definition given for block decomposable hyperplane
arising from two hyperplanes $H_0$ and $H_1$ can be extended by induction to the definition of block decomposable hyperplane $\dec(H_0,\cdots, H_{n-1})$ arising from  $n$ hyperplanes $H_i$ ($0\leq i\leq n-1$), where $H_i$ is a hyperplane of $\cG_3(V_i)$ and $V=\oplus_i V_i.$

\bigskip

In general, given two linear subspaces $V_0$, $V_1$ of $V$ such that $V=V_0\oplus V_1$,
and given two hyperplanes $H_0$ and $H_1$ of $\cG_3(V_0)$
and $\cG_3(V_1)$, there exist several
possible hyperplanes $H$ of $\cG_3(V_0\oplus V_1)$ which are block decomposable and arise from $H_0$ and $H_1$,
namely all of those induced by the forms $h_{\alpha,\beta}:=\alpha h_0+\beta h_1$
with $\alpha,\beta\in\KK\setminus\{0\}$. Even if all these hyperplanes are in general neither
equivalent nor nearly equivalent, they turn out to be always geometrically equivalent and
their geometry of poles depends only on the geometries of $H_0$ and $H_1$.

%Let $H$ be a hyperplane of $\cG_k$.
%Consider  the family ${\mathcal N}(H)$ of all subspaces $S$ of $\cG_k$ such that
%\begin{equation} \{ Y\in\cG_{k-1}: \forall X\in\cG_k \text{ with }
%  Y\leq X \text{ we have } X\in S \}=R^{\uparrow}(H).
%\end{equation}
%Since $H\in{\mathcal N}(H)$, this family is non-empty and, clearly, it is closed under
%intersections of its elements.
%We define
%the \emph{core} $C(H)$ of $H$ as the minimum element of ${\mathcal N}(H)$, i.e.
%$C(H):=\bigcap_{S\in{\mathcal N}(H)}S$.
%%This observation motivates us to give the following definition. For each hyperplane $H$ of $\cG_k(V)$, we call %\emph{core} of $H$ the smallest
%%subspace $C(H)$ of $\cG_k$ such that

%It is immediate to see that given two forms $h_{\alpha,\beta}$ and
%$h_{\alpha',\beta'}$ with $(\alpha,\beta)\neq (\alpha',\beta')$ the corresponding
%block decomposable hyperplanes, say $H$ and $H'$ are different;
%however
%they have the same upper radical and so $C(H)=C(H)'\subset H\cap H'$.
%We will further develop the concept of core of a hyperplane in a future work.

\section{Characterization of the geometry of poles}
\label{sec4}
In this section we will prove our Theorems~\ref{main theorem 1},~\ref{ts} and~\ref{main theorem 2}. As in  the previous sections, let $E:=(e_i)_{i=1}^n$ be a given basis of $V.$
Let $H$ be a given hyperplane of $\cG_3(V)$ and $\cP(H)=(P(H),R^{\uparrow}(H))$ be the geometry of poles of $H.$
In Section~\ref{determination of points and lines} we have explained how to algebraically describe the pointset $P(H)$ and the lineset $R^{\uparrow}(H)$ of the geometry of poles associated to $H.$
The main steps to describe $\cP(H)$ are the following:
\begin{itemize}
\item Consider the bilinear form $\chi_u$ associated to the trilinear form defining $H$ and write the matrix $M_u$ representing $\chi_u$. This is done in Table~\ref{tab2} for forms of rank up to $6$ and in Table~\ref{tab3} for forms of rank $7.$ Recall that for $\dim(V)\leq 7$ and $\KK$ perfect with cohomological dimension at most $1$, all trilinear forms are classified: they are listed in Table~\ref{Tab F}.
\item By Theorem~\ref{lp}, we know that the set of poles is either the pointset of $\PG(V)$ or an algebraic variety. If $\dim(V)=6$ all points of $\PG(V)$ are poles, for any hyperplane $H$ of $\cG_3(V).$ If $\dim(V)=7$, to get the equations describing the variety $P(H)$ we rely on Corollary~\ref{corollary- H pole}, which gives algebraic conditions on $(u_i)_{i=1}^n$ for $[(u_i)_{i=1}^n]$ to be a $H$-pole.
  In the second column of Table~\ref{tab_n7} we have written down those equations, according to the type of $H.$
  \item To describe the lines of $\cP(H)$ we rely on the last part of Section~\ref{determination of points and lines}. In particular, $\ell:=[x,y]\in R^{\uparrow}(H)$ if and only if the functional $\tilde{h}_{xy}$ described in Equation~\eqref{eq2} is the null functional.
This immediately reads as some linear equations in the Pl\"ucker coordinates $|x,y|_{ij}$ of the line $\ell$.
The results of these (straightforward) computations are reported  in the third column of Tables~\ref{tab_n6}
 for forms of rank at most $6$ and Table~\ref{tab_n7} for forms of
rank $7$.
\end{itemize}
So, Tables~\ref{Tab F}, \ref{tab2}, \ref{tab3}, \ref{tab_n6}, \ref{tab_n7} provide an algebraic description of points and lines of $\cP(H),$ for any hyperplane $H$ of $\cG_3(V).$
In the remainder of this section we shall
also provide a geometrical description of $\cP(H).$

\subsection{Hyperplanes arising from forms of rank at most $6$}
\subsubsection{Hyperplanes of type $T_1$}
A hyperplane $H$ of $\cG_3(V)$ of type $T_1$ is defined by a trilinear form of rank $3$ equivalent to $h=\underline{123}$, see the first row of Table~\ref{Tab F}.

Suppose $\dim(V)\geq 4$ and let $(e_i)_{i=1}^n$ be a given basis of $V.$ Then $\Rad(h)=\langle e_i\rangle_{i\geq 4}.$
According to Section~\ref{extensions}, $H$ is a trivial hyperplane $\ext_{\Rad(h)}(\emptyset)$ centered at $\Rad(h).$
By Proposition~\ref{ext hyp sing}, the set of $\ext_{\Rad(h)}(\emptyset)$-poles is the whole pointset of $\PG(V)$ and the lines of the geometry of poles, i.e. the elements in the upper radical $R^{\uparrow}(\ext_{\Rad(h)}(\emptyset))$, are those lines of $\PG(V)$ meeting $\Rad(h)$ non-trivially.
When $n=\dim(V)\leq 6$, this proves part~\ref{m1p1} of Theorem~\ref{main theorem 1}.

  \subsubsection{Hyperplanes of type $T_2$}
A hyperplane $H$ of $\cG_3(V)$ of type $T_2$ is defined by a trilinear form of rank $5$ equivalent to $h=\underline{123}+\underline{145}$, see the second row of Table~\ref{Tab F}.

Suppose $\dim(V)>5.$ Then $\dim(\Rad(h))\geq 1.$ Let $V=\Rad(h)\oplus V'.$
By Section~\ref{extensions}, $H$ is a trivial  extension $\ext_{\Rad(h)}(H')$ of a hyperplane $H'$ of $\cG_3(V').$ By Theorem~\ref{ext hyp th}, we can assume without loss of generality $V'=\langle e_i\rangle_{i=1}^5.$ Put $V_0:=\langle e_i\rangle_{i=2}^5$ and consider the hyperplane $H_0$ of $\cG_2(V_0)$ defined by the functional $\underline{23}+\underline{45}.$
By Section~\ref{expansions}, $H'$ is the expansion $\expa(H_0)$ of $H_0.$ By Subsection~\ref{symp hyp sec}, since the geometry $\cS(H_0)=(G_1(V_0), H_0)$ is a non-degenerate symplectic polar space, the geometry of poles of $\expa(H_0)$ coincides precisely with $\cS(H_0)$. Hence, if $\dim(V)>5$, $H$ is a trivial extension $\ext_{\Rad(h)}(\expa(H_0))$ of a symplectic hyperplane $\expa(H_0).$ The $H$-poles are all the points of $\PG(V)$ and the lines of the geometry of poles are all the lines $\ell$ of $\PG(V)$ meeting $\Rad(h)$ non-trivially or such that $\pi(\ell)\in H_0$ where $\pi$ is the projection onto $V_0$ along $\Rad(h).$

If $\dim(V)=5$ then $H$ is the expansion $\expa(H_0)$  of a non-degenerate symplectic polar space $\cS(H_0)$ defined by the functional $\underline{23}+\underline{45}$ in $V_0=\langle e_i\rangle_{i=2}^5$. By Subsection~\ref{symp hyp sec}, the geometry of poles of $H$ coincides with the symplectic polar space $\cS(H_0).$

Part~\ref{m1p2} of Theorem~\ref{main theorem 1} is proved.

\subsubsection{Hyperplanes of type $T_3$}
A hyperplane $H$ of $\cG_3(V)$ of type $T_3$ is defined by a trilinear form of rank $6$ equivalent to $h=\underline{123}+\underline{456}$, see the third row of Table~\ref{Tab F}.

Suppose $\dim(V)=6$. Let $(e_i)_{i=1}^6$ be a given basis of $V$. Put $V_0=\langle e_1,e_2,e_3\rangle$ and $V_1=\langle e_4,e_5,e_6\rangle$. Clearly, $V=V_0\oplus V_1.$ For $i=0,1$, denote by $\bar{h}_i:=\underline{(3i+1)(3i+2)(3i+3)}$ the trilinear form induced by the restriction of $h$ to $\bigwedge^3 V_i.$ By Section~\ref{decomposable}, $H$ is a decomposable hyperplane $\dec(H_0,H_1)$ of $\cG_3(V)$ arising from the hyperplanes $H_i$, $i=0,1$, of $\cG_3(V_i)$ defined by the forms $\bar{h}_i.$
Since $R^{\uparrow}(H_i)=\emptyset$, by Theorem~\ref{t3},
all points of $\PG(V)$ are elements of the geometry of poles $\cP(H)$ and the lines of
$\cP(H)$ are exactly those lines of $\PG(V)$ intersecting both $\PG(V_0)$ and $\PG(V_1).$
This proves part~\ref{m1p3} of Theorem~\ref{main theorem 1}.

Suppose $\dim(V)>6$. Let $(e_i)_{i\geq 1}$ be a given basis of $V$. Then $\Rad(h)=\langle e_i\rangle_{i\geq 7}.$ By last part of Section~\ref{extensions}, $H$ is a trivial extension $\ext_{\Rad(h)}(\dec(H_0,H_1))$ of a decomposable hyperplane  $\dec(H_0,H_1)$ of $\cG_3(V')$ where $V=\Rad(h)\oplus V'$ and $H_i$, $i=0,1$, are the hyperplanes of $\cG_3(V_i)$, $V_i=\langle e_{3i+1},e_{3i+2}, e_{3i+3} \rangle$, $V'=V_1\oplus V_2$, defined by  $\underline{(3i+1)(3i+2)(3i+3)}$.

\subsubsection{Hyperplanes of type $T_4$}
A hyperplane $H$ of $\cG_3(V)$ of type $T_4$ is defined by a trilinear form of rank $6$ equivalent to $h=\underline{162}+\underline{243}+\underline{135}$, see the fourth row of Table~\ref{Tab F}.

Suppose $\dim(V)=6.$ For any $u\in V$, $\rank(M_u)\leq 4$ (see Table~\ref{tab2} for the description  of the matrix $M_u$). If $(e_i)_{i=1}^6$ is a basis of $V$ and
$V=V_0\oplus V_1$ with
$V_0=\langle e_1,e_2,e_3\rangle$ and $V_1=\langle e_4,e_5,e_6\rangle$, by a direct computation we have that the elements of $V_1$ are poles of degree $3$ while all remaining poles have degree $1$.

Let $\ell=[u,v]$ be a line of $\PG(V)$. By the forth row of Table~\ref{tab_n6}, $\ell\in R^{\uparrow}(H)$ if and only if its Pl\"{u}cker coordinates satisfy $6$ linear equations. More explicitly, we have that $\ell=[u,v]\in R^{\uparrow}(H)$ if and only if $u=\bar{u}_0+\bar{u}_1$ and $v=\omega(\bar{u}_0)\in V_1$ with $\bar{u}_0\in V_0$, $\bar{u}_1\in V_1$ and
$\omega\colon V \rightarrow V$,
$\omega(u_1,u_2,u_3,u_4,u_5,u_6)=(u_4,u_5,u_6,u_1,u_2,u_3).$
Note that $\omega$ interchanges $V_0$ and $V_1$.
Indeed,  if  $u=\bar{u}_0+\bar{u}_1$ and $v=\bar{v}_0+\bar{v}_1$ with $\bar{u}_0, \bar{v}_0\in V_0$ and $\bar{u}_1, \bar{v}_1\in V_1$, the Pl\"{u}cker coordinates of the line $\ell=[u,v]$, satisfy the equations $|u,v|_{12}=0, |u,v|_{13}=0,|u,v|_{23}=0$ if and only if $\bar{v}_0=\lambda \bar{u}_0$ with $\lambda\in \KK.$
Hence $\ell=[u,v]=[u,v-\lambda u]=[\bar{u}_0+\bar{u}_1, \bar{v}_1']$ with $\bar{v}_1'=\bar{v}_1-\lambda \bar{u}_1\in V_1.$

The remaining three equations $|x,y|_{26}-|x,y|_{35}=0, |x,y|_{16}-|x,y|_{34}=0, |x,y|_{24}-|x,y|_{15}=0$ are satisfied by the  Pl\"{u}cker coordinates of $\ell=[u,v]=[\bar{u}_0+\bar{u}_1, \bar{v}_1']$ if and only if either $\bar{u}_0$ is the null vector and in this case $\ell=[\bar{u}_1, \bar{v}_1']\subset V_1$ or $[\bar{v}
_1']=[(0,0,0,u_1,u_2,u_3)]$ where  $[u]=[(u_1,u_2,u_3,u_4,u_5,u_6)].$
\par
This proves part~\ref{m1p4} of Theorem~\ref{main theorem 1}.

Suppose $\dim(V)>6$. Let $(e_i)_{i\geq 1}$ be a given basis of $V$. Then $\Rad(h)=\langle e_i\rangle_{i\geq 7}.$ By last part of Section~\ref{extensions}, $H$ is a trivial extension $\ext_{\Rad(h)}(H')$ of a hyperplane $H'$ of $\cG_3(V')$ where $V=\Rad(h)\oplus V'$, $V'=\langle e_i\rangle_{i=1}^6$ and $H'$ is defined by a trilinear form equivalent to
$\underline{162}+\underline{243}+\underline{135}.$
A description of the geometry of poles of $H$ follows from Theorem~\ref{extended hyperplane} and the already done case of hyperplanes of type $T_4$ of $\cG_3(V)$ for $\dim(V)=6.$

\subsubsection{Hyperplanes of type $T_{10,\lambda}^{(i)}$}
A hyperplane $H$ of $\cG_3(V)$ of type $T_{10,\lambda}^{(i)}$ is defined by a trilinear form of rank $6$ as written in row $10$ or $11$ of Table~\ref{Tab F}, according as $\chr(\KK)$ is odd or even.
We remark that forms of type $T_{10,\lambda}^{(1)}$ make sense also in even characteristic, provided
that the field $\KK$ is not perfect.

Let $\dim(V)=6.$ For any $u\in V$, the matrix $M_u$ representing the bilinear alternating form $\chi_u$ associated to $H$ is written in Table~\ref{tab2}. We have that $\rank(M_u)\leq 4$ for every $u\in V$, i.e. every point of $\PG(V)$ is a pole. Actually, we will prove that $\rank(M_u)=4$ for any $u\in V$, i.e. every point of $\PG(V)$ is a pole of degree $1$, equivalently, $R^{\uparrow}(H)$ is a line spread of $\PG(V).$

Consider first a hyperplane of type $T_{10,\lambda}^{(1)}.$ Suppose by way of contradiction that  $\rank(M_u)<4,$ i.e. all minors of order 4 vanish. With $u=(u_1,\ldots,u_6)$, take the three $4\times 4$ principal
minors of $M_u$ given by
\[ \begin{array}{c|c}
     \multicolumn{2}{c}{\text{Case $i=1$}} \\[5pt]
     \begin{array}{c}
      \text{Principal submatrix of $M_u$} \\
       \text{corresponding to rows/columns}
     \end{array} & \text{Value of the minor} \\ \hline
     1,2,4,5 & \lambda^2(\lambda u_6^2-u_3^2)^2 \\
     1,3,4,6 & \lambda^2(\lambda u_5^2-u_2^2)^2 \\
     2,3,5,6 & \lambda^2(\lambda u_4^2-u_1^2)^2.
   \end{array} \]
By the third column of Table~\ref{Tab F} corresponding to $T_{10,\lambda}^{(1)}$, we have that $p_{\lambda}(t)=t^2-\lambda$ is an irreducible polynomial in $\KK[t].$
Hence the minors mentioned are null if and only if $u_i=0$ for all $i=1,\dots, 6.$

 We argue in a similar way for case
 $T_{10,\lambda}^{(2)}$ with $\chr(\KK)=2$,
 by choosing the minors of $M_u$ given by
 \[ \begin{array}{c|c}
      \multicolumn{2}{c}{\text{Case $i=2$}} \\[5pt]
          \begin{array}{c}
      \text{Principal submatrix of $M_u$} \\
       \text{corresponding to rows/columns}
     \end{array} & \text{Value of the minor} \\ \hline
     1,2,4,5 & (u_6^2+\lambda u_3u_6+u_3^2)^2 \\
     1,3,4,6 & (u_5^2+\lambda u_2u_5+u_2^2)^2 \\
     2,3,5,6 & (u_4^2+\lambda u_1u_4+u_1^2)^2.
   \end{array} \]
 By the third column of Table~\ref{Tab F} corresponding to $T_{10,\lambda}^{(2)}$, we have that  $p_{\lambda}(t):=t^2+\lambda t+1$ is irreducible in  $\KK[t]$; hence
 $\rank M_u=4$ unless $u=(0,\ldots,0)$.
 So $R^{\uparrow}(H)$ is a line-spread of $\PG(V)$.

\begin{lemma}
  \label{ldes}
  Let $H$ be a hyperplane of $\cG_3(V)$ with $\dim V=6$ whose upper radical
  $R^{\uparrow}(H)$ is a line-spread of $\PG(V)$.
  Then $R^{\uparrow}(H)$ is a Desarguesian line-spread of $\PG(V)$.
\end{lemma}
\begin{proof}
For simplicity of notation, denote by $\cS$ the line-spread of $\PG(V)$ induced by $H.$
 Then, by duality, $H$ induces also a line
  spread $\cS^*$ in the dual space  $\PG(V^*)$, where $V^*$ is the dual of $V$.
  In particular, a  $4$-dimensional vector space $\Sigma$ is in $\cS^*$
  if and only if all planes contained
  in $\Sigma$ are elements of $H$.
We will prove that $\cS$ is a {\it normal spread}, i.e. given any
  two distinct elements $\ell_1,\ell_2\in\cS$ and $\Sigma=\ell_1+\ell_2$,
   the set $\cS_{\Sigma}:=\{ \ell\in\cS: \ell\subseteq\Sigma \}$
   is a line-spread of $\Sigma$.

   Let $\pi$ be a plane contained in $\Sigma$. Then,
   $\pi\cap\ell_1\neq\emptyset\neq\pi\cap\ell_2$ and
   we can write $\pi=[p_1,p_2,q_1+q_2]$ with
   $p_1,q_1\in\ell_1$, $p_2,q_2\in\ell_2$ suitably chosen.
   Denote by $h$ a form defining the hyperplane  $H$.
   Then $h(p_1\wedge p_2\wedge (q_1+q_2))=h(p_1\wedge p_2\wedge q_1)+h(p_1\wedge p_2\wedge q_2)=0$
   as $h$ is identically zero on all planes through either $\ell_1$ or $\ell_2$.
   Hence $\pi\in H$ and, consequently, $\Sigma\in\cS^*$.

   So the $3$-dimensional projective space spanned by any two elements of $\cS$ is in $\cS^*$;
   by duality, the intersection of any two elements of $\cS^*$ is in $\cS$.
   Take now $\Sigma\in\cS^*$ and $p\in\Sigma$; denote by $\ell_p$
   the unique line of $\cS$ with $p\in\ell_p$.
   Let $\ell'\in\cS$ such that $\ell'$ is not
   contained in $\Sigma$. Then, $\Sigma'=\ell_p+\ell'\in\cS^*$ and, by the
   argument above, $\Sigma'\cap\Sigma\in\cS$.
   Since $p\in\Sigma'\cap\Sigma$, it follows that $\Sigma'\cap\Sigma=\ell_p$.
   This for all points $p\in\Sigma$; so
   $\cS_{\Sigma}$ is a spread of $\Sigma$.
Hence $\cS$ is a normal spread and by~\cite[Theorem 2]{BC72} $\cS$ is a Desarguesian line spread of $\PG(V).$
\end{proof}
Part~\ref{m1p5} of Theorem~\ref{main theorem 1} is now proved.

\medskip

\noindent{\bf Proof of Theorem~\ref{ts}.}
By Lemma~\ref{ldes}, line-spreads of $\PG(6,V)$ induced by hyperplanes of $\cG_3(V)$ are Desarguesian.
As Desarguesian line-spreads are coordinatized over division rings, there must exist a division ring ${\mathbb D}$ having dimension $2$ over $\KK$ (see \cite{BB66}) with either ${\mathbb D}$ commutative or
$\KK$ being the center of ${\mathbb D}$.
We show that ${\mathbb D}$ is commutative.
Take $a\in{\mathbb D}\setminus\KK$. Then the algebraic extension
$\KK(a)$ is a proper field extension of $\KK$ contained in ${\mathbb D}$; so
$2=[{\mathbb D}:\KK]\geq[\KK(a):\KK]\geq 2$. It follows that ${\mathbb D}=\KK(a)$ and ${\mathbb D}$
  is a field which is an algebraic extension of degree $2$ of
$\KK$. So, for ${\mathbb D}$ to exist, $\KK$ must not be quadratically closed.
\hspace{10.5cm}$\Box$

\begin{remark}
 The commutativity of ${\mathbb D}$ in the proof of Theorem~\ref{ts} is also a consequence
 of the Artin-Wedderburn theorem. We have provided a short argument.
\end{remark}

Suppose $\dim(V)>6$. Let $(e_i)_{i\geq 1}$ be a given basis of $V$. Then $\Rad(h)=\langle e_i\rangle_{i\geq 7}.$ By last part of Section~\ref{extensions}, $H$ is a trivial extension $\ext_{\Rad(h)}(H')$ of a hyperplane $H'$ of $\cG_3(V')$ where $V=\Rad(h)\oplus V'$, $V'=\langle e_i\rangle_{i=1}^6$ and $H'$ is defined by a trilinear form of type $T_{10,\lambda}^{(i)}.$ A description of the geometry of poles of $H$ follows from Theorem~\ref{extended hyperplane} and the already done case of hyperplanes of type $T_{10,\lambda}^{(i)}$ of $\cG_3(V)$ with $\dim(V)=6.$

\subsection{Hyperplanes arising from forms of rank $7$}
Throughout this section let $[u]=[(u_i)_{i=1}^7]$.
\subsubsection{Hyperplanes of type $T_5$}
   A hyperplane $H$ of $\cG_3(V)$ of type $T_5$ is defined by a trilinear form of rank $7$ equivalent to $h=\underbar{123}+\underbar{456}+\underbar{147}$, see the fifth  row of Table~\ref{Tab F}.

Suppose $\dim(V)=7.$
Straightforward computations shows that $\rank(M_u)\leq 4$ (see Table~\ref{tab3} for the description  of the matrix $M_u$)
if and only if $u_1=0$ or $u_4=0$ and $\rank(M_u)=2$ if and only if $u_1=u_4=u_5=u_6=0$ or $u_1=u_2=u_3=u_4=0.$

Let $\cS_1$ and $\cS_2$ the hyperplanes of $\PG(V)$ with equations respectively $u_1=0$ and $u_4=0$.
Denoted by $P(H)$ the set of poles of $H$, we then have $P(H)=\cS_1\cup \cS_2.$
A point $[u]$ has degree $4$ if and only if $[u]\in A_1\cup A_2$ where $A_1$ is the plane of $\PG(V)$ of equation
$u_1=u_4=u_5=u_6=0$ and $A_2$ is the plane of $\PG(V)$ of equation
$u_1=u_2=u_3=u_4=0.$

  Take $[u]=[(0,u_2,\ldots,u_7)]\in\cS_1$. In this case,
  by the equations of Table~\ref{tab_n7}, a line
  $\ell:=[u,v]$ of $\PG(V)$ through $[u]$ is in $R^{\uparrow}(H)$ if and only if $\ell\subseteq \cS_1$, intersects $A_1$ non-trivially and it is totally isotropic for the non-degenerate bilinear alternating form of $\cS_1$ defined by $\beta(u,v)=u_2v_3-u_3v_2+u_4v_7-u_7v_4+u_5v_6-u_6v_5.$
  A similar argument shows that if $[u]$ is taken in $\cS_2$, then
  any line of $R^{\uparrow}(H)$ through it meets $A_2$ and it is totally isotropic for the non-degenerate bilinear alternating form of $\cS_2$ defined by $\beta(u,v)=u_1v_7-u_7v_1+u_2v_3-u_3v_2+u_5v_6-u_6v_5.$

This proves part~\ref{m2p1} of Theorem~\ref{main theorem 2}.

\subsubsection{Hyperplanes of type $T_6$}
   A hyperplane $H$ of $\cG_3(V)$ of type $T_6$ is defined by a trilinear form of rank $7$ equivalent to $h=\underbar{152}+\underbar{174}+\underbar{163}+\underbar{243}$, see the sixth  row of Table~\ref{Tab F}.
   Suppose $\dim V=7$; straightforward computations show that $\rank(M_u)\leq 4$ (see Table~\ref{tab3}
   for the description of the matrix $M_u$) if and only if $u_1=0$.
   Let $\cS$ be the hyperplane of $\PG(V)$ of equation $u_1=0$. Then the set of the $H$-poles is
   precisely the point-set of $\cS$.
   Also, a point $[u]$ has degree $4$ if and only if $[u]\in A$ where $A$ is the plane of
   equation $u_1=u_2=u_3=u_4=0$.
   Take $u=[(0,u_2,\ldots,u_7)]\in\cS$. By the equations
   $|u,v|_{34}=|u,v|_{24}=|u,v|_{23}$,
   of  Table~\ref{tab_n7}, each line of the upper radical
   must
   intersect the plane $A$.
   By the equation $|u,v|_{25}+|u,v|_{36}+|u,v|_{47}=0$, we have that a line $\ell=[u,v]$
   is in $R^{\uparrow}(H)$ if and only if $\ell\subseteq\cS$, $\ell\cap A\neq0$ and
   $\ell$ is totally isotropic for the non-degenerate alternating form
   $\beta(u,v):=|u,v|_{25}+|u,v|_{36}+|u,v|_{47}=0$.
     This proves part~\ref{m2p2} of Theorem~\ref{main theorem 2}.

  \subsubsection{Hyperplanes of type $T_7$}
  A hyperplane $H$ of $\cG_3(V)$ of type $T_7$ is defined by a trilinear form of rank $7$ equivalent
 to $h=\underbar{146}+\underbar{157}+\underbar{245}+\underbar{367},$
 see the seventh  row of Table~\ref{Tab F}.
 Suppose $\dim V=7$; straightforward computations show that $\rank(M_u)\leq 4$ (see Table~\ref{tab3}
 for the description of the matrix $M_u$) if and only if $u_5u_7+u_4u_6=0$.
 Let $A$ be the plane of $\PG(V)$ of equations $u_4=u_5=u_6=u_7=0$ and denote by $\cQ$ the hyperbolic
 quadric of equation $u_5u_7+u_4u_6=0$ embedded in the subspace $W$ of equations $u_1=u_2=u_3=0$.
 The set of $H$-poles is the point-set of the quadratic cone of $\PG(V)$ with vertex $A$ and
 basis $\cQ$. Also, a point $[u]$ has degree $4$ if and only if $[u]\in\cC$ where $\cC$ is
 the conic of $A$ with equation $u_1^2-u_2u_3=0$.
 Using the equations of Table~\ref{tab_n7}, it is possible to associate to any
 $[x]=[(st,t^2,s^2,0,0,0,0)]\in\cC$
 a unique line $\ell_x=[(0,0,0,s,0,0,t),(0,0,0,0,s,-t,0)]$ of $\cQ$. \par

 With $[x]\in\cC$, denote by $\Res_{\langle A,\ell_x\rangle}(x)$ the projective geometry whose points are
 all the $2$-dimensional vector spaces
 through $x$ in the $5$-dimensional vector space $\langle A,\ell_x\rangle$ and whose
 lines are the $3$-dimensional vector spaces of $\langle A,\ell_x\rangle$ through $x$.
 Note that $\Res_{\langle A,\ell_x\rangle}(x)\cong\PG(3,\KK)$.

 \begin{prop}
   For any $[x]\in\cC$ there exists a line spread $\cS_x$ of $\Res_{\langle A,\ell_x\rangle}(x)$ such that
   $\ell\in R^{\uparrow}(H)$ if and only if $\ell\subseteq\langle x,s\rangle$ for some $x\in\cC$ and
   $s\in\cS_x$.
 \end{prop}
\begin{proof}
  Suppose $x,x'\in\cC$ with $x\neq x'$. Then, $\langle A,\ell_x\rangle\cap\langle A,\ell_{x'}\rangle=A$.
  We shall now define $\cS_x$ as follows
  \[ \cS_x:=\{ \pi_p: p\in\langle A,\ell_x\rangle\setminus A \}\cup\{A\} \]
  where $\pi_p$ is the unique plane of $\PG(V)$ spanned by all the lines through $p$ in
  $R^{\uparrow}(H)$. Note that $\delta(p)=2$ and $x\in\pi_p\subseteq\langle A,\ell_x\rangle$.

  Let now $[q]$ be a point in $\pi_p$ with $\pi_p\in\cS_x$.
  If $[q]\not\in [x,p]$, then $[x,q]$ and $[p,q]$ are
  both lines in the upper radical of $H$ through $[q]$; since $\delta(q)=2$ we have $\pi_q=\pi_p$.
  If $[q]\in[x,p]$ we consider a point $[r]\not\in [x,p]$ and apply the same argument to show
  that $\pi_q=\pi_r=\pi_p$. Hence all lines of $\PG(V)$ in any plane $\pi_p\in\cS_x$ are in the upper
  radical and for any pole $p\not\in A$, the lines of the upper radical through $p$ are contained
  in $\pi_p$.

  Suppose $\pi_p$ and $\pi_q$ are two lines of $\cS_x$ with non-trivial intersection, i.e.
  $\pi_p\cap\pi_q$ is a line through $x$ not contained in $A$.
  Any point on this line would have degree at least $3$ --- a contradiction.

  Finally all lines contained in $A$ are elements of the upper radical and $A$ reads as
  one line of $\cS_x$.
\end{proof}
  This proves part~\ref{m2p3} of Theorem~\ref{main theorem 2}.

  \subsubsection{Hyperplanes of type $T_8$}
   A hyperplane $H$ of $\cG_3(V)$ of type $T_8$ is defined by a trilinear form of rank $7$ equivalent to $h=\underbar{123}+\underbar{145}+\underbar{167}$, see the eighth  row of Table~\ref{Tab F}.
   Suppose $\dim V=7$. Put $V_0=\langle e_i\rangle_{i=2}^7$ and consider the hyperplane $H_0$ of
   $\cG_2(V_0)$ define by the functional $\underbar{23}+\underbar{45}+\underbar{67}$.
   By Subsection~\ref{expansions}, $H$ is the expansion $\expa(H_0)$ of $H_0$
   and
    since the geometry $S(H_0):=(G_1(V_0),H_0)$ is a
   non-degenerate symplectic polar space, the geometry of poles of $\expa(H_0)$
   coincides with $S(H_0)$.
\par
  This proves part~\ref{m2p4} of Theorem~\ref{main theorem 2}.

  \subsubsection{Hyperplanes of type $T_9$ and $T_{12,\mu}$}
  Hyperplanes $H$ of $\cG_3(V)$ of types either $T_9$ or $T_{12,\mu}$ are defined by trilinear forms of rank $7$ nearly equivalent to $h=\underbar{123}+\underbar{456}+\underbar{147}+\underbar{257}+\underbar{367}$, see rows $9$ and $14$ of Table~\ref{Tab F}.
  Suppose $\dim V=7$; straightforward computations show that $\rank(M_u)\leq 4$ (see Table~\ref{tab3}
 for the description of the matrix $M_u$) if and only if $[u]$ is a
 point of a non-degenerate parabolic quadric $\cQ$ of $\PG(V)$. More precisely,
 each point of $\cQ$ has degree $2$.
  % In~\cite[\S 3.9, 3.12]{CH88} there is a detailed study of the group
  % stabilizing a form of these types, which is the product of the
  % Chevalley group of type $G_2$ over $\KK$ with the subgroup of
  % the cubic roots of unity in $\KK$.

 The equations appearing in
  Table~\ref{tab_n7} for cases $T_9$ (and $T_{12,\mu}$) are
  the same as those in \cite[\S 2.4.13]{VM98} for the standard embedding
  of the Split-Cayley  hexagon $H(\KK)$ in $\PG(6,\KK)$;
  see also \cite{Ti59}.
  Hence, the geometry of the poles of $H$ is precisely a Split-Cayley hexagon.
  For these reasons, hyperplanes of this type are called
  \emph{hexagonal}.
  This proves part~\ref{m2p5} of Theorem~\ref{main theorem 2}.

  \subsubsection{Hyperplanes of type $T_{11,\lambda}^{(i)}$}
  A hyperplane $H$ of $\cG_3(V)$ of type $T_{11,\lambda}^{(i)}$ ($i=1,2$) is defined by a trilinear form of rank $7$ as written in row $12$ or $13$ of Table~\ref{Tab F}, according as $\chr(\KK)$ is odd or even.
  As in the case $T_{10,\lambda}^{(1)}$, we remark that forms of type
  $T_{11,\lambda}^{(1)}$ may also be considered in even characteristic,
  provided that the field is not perfect.
  The geometries of poles
 arising in both cases $i=1,2$ afford a similar description.
 \par
 Suppose $\dim V=7$.
 By the proof of Theorem~\ref{lp}, straightforward computations show that $\rank(M_u)\leq 4$ if and only if $\det(M_u^{(7)})/u_7^2=0$ where
by $M_u^{(7)}$ is the submatrix of $M_u$ (see Table~\ref{tab_n7} for the
description of $M_u$) obtained by deleting its last row and column.

First consider the case $T_{11,\lambda}^{(1)}$. Then $[u]$ is a pole if
and only if its coordinates satisfy the equation $\lambda u_4^2-u_1^2=0$.
Since the polynomial $p_{\lambda}(x)=x^2-\lambda$ is irreducible in $\KK$,
the points satisfying the above equation have coordinates with $u_1=u_4=0$.
Hence, the set of poles is $\cS:=\cS_1\cap\cS_2$ where $\cS_1$ is the hyperplane
of $\PG(V)$ of equation $u_1=0$ and $\cS_2$ is the hyperplane of $\PG(V)$
of equation $u_4=0$.
Considering the $4\times 4$ principal minors of $M_u$ given by
\[ \begin{array}{c|c}
     \multicolumn{2}{c}{\text{Case $i=1$}} \\[5pt]
     \begin{array}{c}
      \text{Principal submatrix of $M_u$} \\
       \text{corresponding to rows/columns}
     \end{array}
 & \text{Value of the minor} \\ \hline
     1,2,4,5 & \lambda^2(\lambda u_6^2-u_3^2)^2 \\
     1,3,4,6 & \lambda^2(\lambda u_5^2-u_2^2)^2, \\
   \end{array} \]
we have that the only point of degree $4$ is $[e_7]$.

In the case $T_{11,\lambda}^{(2)}$, then $[u]$ is a pole if
and only if its coordinates satisfy the equation $u_4^2+\lambda u_1u_4+u_1^2=0$.
Since the polynomial $p_{\lambda}(x)=x^2+\lambda x+1$ is irreducible in $\KK$,
the points satisfying the above equation have coordinates with $u_1=u_4=0$.
Hence, the set of poles is $\cS:=\cS_1\cap\cS_2$ where $\cS_1$ is the hyperplane
of $\PG(V)$ of equation $u_1=0$ and $\cS_2$ is the hyperplane of $\PG(V)$
of equation $u_4=0$.
Considering the $4\times 4$ principal minors of $M_u$ given by
 \[ \begin{array}{c|c}
      \multicolumn{2}{c}{\text{Case $i=2$}} \\[5pt]
     \begin{array}{c}
      \text{Principal submatrix of $M_u$} \\
       \text{corresponding to rows/columns}
     \end{array}
      & \text{Value of the minor} \\ \hline
     1,2,4,5 & (u_6^2+\lambda u_3u_6+u_3^2)^2 \\
     1,3,4,6 & (u_5^2+\lambda u_2u_5+u_2^2)^2 \\
   \end{array} \]
so we have that the only point of degree $4$ is $[e_7]$ as above.

We now provide a geometric description of $R^{\uparrow}(H)$ holding for both
$i=1$ and $i=2$.
Denote by $\Res_{\cS}(e_7)$ the projective geometry whose points are all the
lines of $\cS$ through $[e_7]$ and whose lines are the planes of $\cS$ through
$[e_7]$. It is well-known that $\Res_{\cS}(e_7)\cong\PG(3,\KK)$.
Consider the following set
\[ \cF:=\{\pi_p: p\in\cS\setminus[e_7] \} \]
where $\pi_p$ is the plane of $\PG(V)$ spanned by the lines in $R^{\uparrow}(H)$.
\begin{prop}
  The set $\cF$ is a line-spread of $\Res_{\cS}(e_7)$ and
  \[ R^{\uparrow}(H)=\{ \ell\subseteq\pi_p: \pi_p\in\cF \}. \]
\end{prop}
\begin{proof}
Take $[p]\in\cS\setminus[e_7]$. Since $[p]$ a pole of degree $2$, then there
exists a plane $\pi_p$ with $[p]\in\pi_p$ spanned by lines of $R^{\uparrow}(H)$.
Note that $[p,e_7]\in R^{\uparrow}(H)$; so $[e_7]\in\pi_p$.
Any line $\ell\subseteq\pi_p$ is in the upper radical of $H$. Indeed,
  let now $[q]$ be a point in $\pi_p$ with $\pi_p\in\cF$.
  If $[q]\not\in [x,p]$, then $[x,q]$ and $[p,q]$ are
  both lines in the upper radical of $H$ through $[q]$;
  since $\delta(q)=2$ we have $\pi_q=\pi_p$.
  If $[q]\in[x,p]$ we consider a point $[r]\not\in [x,p]$ and apply the same argument to show
  that $\pi_q=\pi_r=\pi_p$. Hence all lines of $\PG(V)$ in any plane
  $\pi_p\in\cF$ are in the upper
  radical of $H$.

Suppose now $\pi_p,\pi_q\in\cF$ and $\pi_p\cap\pi_q=r$ where $r$ is a
line through $[e_7]$. Then any point on $r$ has degree at least $3$, a contradiction.
We have proved that the set $\cF$ is a line-spread of $\Res_{\cS}(e_7)$.
The characterization of the upper radical is now straightforward.
\end{proof}

  This proves part~\ref{m2p6} of Theorem~\ref{main theorem 2}.

\vskip1cm
 \begin{minipage}[t]{\textwidth}
Authors' addresses:
\vskip.2cm\noindent\nobreak
\centerline{
\begin{minipage}[t]{7cm}
Ilaria Cardinali\\
Department of Information Engineering and Mathematics\\University of Siena\\
Via Roma 56, I-53100, Siena, Italy\\
ilaria.cardinali@unisi.it\\
\end{minipage}\hfill
\begin{minipage}[t]{6.5cm}
Luca Giuzzi\\
D.I.C.A.T.A.M. \\
Section of Mathematics \\
University of Brescia\\
Via Branze 43, I-25123, Brescia, Italy \\
luca.giuzzi@unibs.it
\end{minipage}}
\end{minipage}

\newpage
\appendix
\section{Tables}
\label{appendix}
\label{tables}
%{\large{\bf Table 1: Types for linear functionals on $\bigwedge^3 V$.}}

\begin{table}[h]
\caption{Types for linear functionals on $\bigwedge^3 V$.}

\begin{small}
\[ \begin{array}[t]{c|l|c|l}
     \text{Type} & \text{Description} & \text{Rank} & \text{Special conditions, if any} \\ \hline
{}& & & \\
     T_1 & \underbar{123} & 3 & \\
     T_2 & \underbar{123}+\underbar{145} & 5 & \\
     T_3 & \underbar{123}+\underbar{456} & 6 & \\
     T_4 & \underbar{162}+\underbar{243}+\underbar{135} & 6 &  \\
     T_5 & \underbar{123}+\underbar{456}+\underbar{147} & 7 & \\
     T_6 & \underbar{152}+\underbar{174}+\underbar{163}+\underbar{243} & 7 & \\
    T_7 & \underbar{146}+\underbar{157}+\underbar{245}+\underbar{367} & 7 & \\
    T_8 & \underbar{123}+\underbar{145}+\underbar{167} & 7 & \\
     T_9 & \underbar{123}+\underbar{456}+\underbar{147}+\underbar{257}+\underbar{367} & 7 & \\
 {} & & & \\
     T_{10,\lambda}^{(1)} & \underbar{123}+\lambda(\underbar{156}+\underbar{345}+\underbar{426}) & 6 & p_\lambda(t):=t^2-\lambda \text{ irreducible in } \KK[t].\\
{} & & & \\
T_{10,\lambda}^{(2)} &\begin{array}[t]{@{}l}\underbar{126}+\underbar{153}+\underbar{234}+\\
(\lambda^2+1)\underbar{456}+\lambda(\underbar{156}+\underbar{345}+\underbar{426}) \end{array} & 6 & \begin{array}[t]{@{}l} \chr(\KK)=2 \text{ and}\\
                     p_{\lambda}(t):=t^2+\lambda t+1 \text{ irreducible in } \KK[t].\\
\end{array} \\
{}& & & \\
      T_{11,\lambda}^{(1)} & [\text{the same as at} ~ T^{(1)}_{10,\lambda}] +\underbar{147} & 7 & \text{same conditions as for} ~ T_{10,\lambda}^{(1)}\\
      T_{11,\lambda}^{(2)} & [\text{the same as at} ~ T^{(2)}_{10,\lambda}]+\underbar{147} & 7 & \text{same conditions as for} ~ T_{10,\lambda}^{(2)}\\
     T_{12,\mu} & \mu [\text{the same as at} ~ T_9] & 7 &
     p_{\mu}(t)=t^3-\mu \text{ irreducible in $\KK[t]$ } \\
   \end{array}\qquad
\]
\end{small}
\label{Tab F}
\end{table}

According to the clauses assumed on $\lambda$, types $T^{(r)}_{s,\lambda}$ ($r\in \{1,2\}$, $s\in \{10,11\}$) can be considered only when $\KK$ is not quadratically closed. Moreover, when $\lambda \neq \lambda'$ the types $T^{(r)}_{s,\lambda}$ and $T^{(r)}_{s,\lambda'}$ are different up to linear and
near equivalence, even if they might describe geometrically equivalent forms.

 It follows from Revoy~\cite{Revoy79} and Cohen and Helminck~\cite{CH88} that two functionals of types $T_i$ and $T_j$ with $1\leq i < j \leq 9$ are never nearly equivalent; a functional of type $T_i$ with $i \leq 9$ is never nearly equivalent to a functional of type $T^{(r)}_{s,\lambda}$; two functionals of type $T^{(r)}_{s,\lambda}$ and $T^{(r')}_{s',\lambda'}$ with $(r,s) \neq (r', s')$ are never nearly equivalent while two functionals of type $T^{(r)}_{s,\lambda}$ and $T^{(r)}_{s,\lambda'}$ are nearly equivalent if and only if, denoted by $\mu$ and $\mu'$ respectively a root of $p_\lambda(t)$ and a root of $p_{\lambda'}(t)$ in the algebraic closure of $\KK$, we have $\KK(\mu) = \KK(\mu')$ (see the fourth column of Table~\ref{Tab F} for the definition of $p_{\lambda}(t)$).
The forms $T_{9}$ and $T_{12,\mu}$ are not linearly equivalent; however they
are, by construction, nearly equivalent.

Also,
the forms $T_{10,\lambda}^{(i)}$ and $T_{10,\lambda'}^{(i)}$
as well as $T_{11,\lambda}^{(i)}$ and $T_{11,\lambda'}^{(i)}$ are not
in general nearly-equivalent however they are geometrically equivalent.
In particular both $T_{10,\lambda}^{(1)}$ and $T_{10,\lambda'}^{(2)}$  induce a Desarguesian
 spread on $\PG(V)$.
Note that the forms $T_{10,\lambda}^{(1)}$ exist only if $\chr(\KK)\neq2$ or
if $\chr(\KK)=2$ and $\KK$ is not perfect.
 The forms of type $T_{10,\lambda}^{(2)}$ are equivalent to form of type $T_{10,\lambda}^{(1)}$
 if $\chr(\KK)\neq 2$; however, they are
not equivalent if $\chr(\KK)=2$.

%{\large{\bf Table 2: Matrices associated to forms of rank up to $6$.}}

\begin{table}[ht]
  \caption{Matrices associated to forms of rank up to $6$.}
    \[\arraycolsep=1pt
      \begin{array}[t]{l|c|l}
        \text{Type\ } & \text{Rk} & \text{Matrix} \\ \hline
        T_1 & 3 &
            {\tiny \begin{pmatrix}
            0 & u_3 & -u_2 \\
            -u_3 & 0 & u_1 \\
            u_2 & -u_1 & 0
          \end{pmatrix}} \\
        T_2 & 5 &
                  {\tiny \begin{pmatrix}
                      0 & u_3 & -u_2 & u_5 & -u_4 \\
                      -u_3& 0   & u_1 & 0    & 0    \\
                      u_2 & -u_1 &0 & 0 & 0 \\
                      -u_5 & 0 & 0 & 0 & u_1 \\
                      u_4 & 0 & 0 & -u_1 & 0
                    \end{pmatrix}} \\
        T_3 & 6 &
                  {\tiny\begin{pmatrix}
                      0 & u_3 & -u_2 & 0 & 0 & 0 \\
                      -u_3 & 0 & u_1 & 0 & 0 & 0 \\
                      u_2 & -u_1 & 0 & 0 & 0 & 0 \\
                      0 & 0 & 0 & 0 & u_6 &-u_5 \\
                      0 & 0 & 0 & -u_6 & 0 & u_4 \\
                      0 & 0 & 0 & u_5 & -u_4 & 0
                    \end{pmatrix}} \\
      \end{array}\quad
      \begin{array}[t]{l|c|l}
        \text{Type\ } & \text{Rk} & \text{Matrix} \\ \hline
      T_4 & 6 & {\tiny\begin{pmatrix}
          0 & -u_6 & u_5 & 0 & -u_3 & u_2 \\
          u_6 & 0 & -u_4 &u_3 & 0 & -u_1 \\
          -u_5 & u_4 & 0 & -u_2 & u_1 & 0 \\
          0 & -u_3 & u_2 & 0 & 0 & 0 \\
          u_3 & 0 & -u_1 & 0 & 0 & 0 \\
          -u_2 & u_1 & 0 & 0 & 0 & 0
        \end{pmatrix}} \\
        T_{10,\lambda}^{(1)} & 6 & {\tiny
        \begin{pmatrix}
          0 & u_3 & -u_2 & 0 & \lambda u_6 & -\lambda u_5 \\
          -u_3 & 0 & u_1 & -\lambda u_6 & 0 & \lambda u_4 \\
          u_2 & -u_1 & 0 & \lambda u_5 & -\lambda u_4 & 0 \\
          0 & \lambda u_6 & -\lambda u_5 & 0 & \lambda u_3 & -\lambda u_2 \\
          -\lambda u_6 & 0 & \lambda u_4 & -\lambda u_3 & 0 & \lambda u_1 \\
          \lambda u_5 & -\lambda u_4 & 0 & \lambda u_2 & -\lambda u_1 & 0
        \end{pmatrix}}
      \end{array} \\
    \]
    \[\arraycolsep=1pt
      \begin{array}[t]{l|c|l}
        \text{Type\ } & \text{Rk} & \text{Matrix} \\ \hline
          T_{10,\lambda}^{(2)} & 6 & {\tiny\begin{pmatrix}
      0 & u_6 & -u_5 & 0 & \lambda u_6+u_3 & -\lambda u_5-u_2  \\
      -u_6 & 0 & u_4 & -\lambda u_6-u_3 & 0 & \lambda u_4+u_1  \\
      u_5 & -u_4 & 0 & \lambda u_5+u_2 & -\lambda u_4-u_1 & 0  \\
      0 & \lambda u_6+u_3 & -\lambda u_5-u_2 & 0 & (\lambda^2+1)u_6+\lambda u_3 & (-\lambda^2-1)u_5 -\lambda u_2  \\
      -\lambda u_6-u_3 & 0 & \lambda u_4+u_1 & -(\lambda^2+1)u_6-\lambda u_3 & 0 & (\lambda^2+1)u_4+\lambda u_1  \\
      \lambda u_5+u_2 & -\lambda u_4-u_1 & 0 & (\lambda^2+1)u_5+\lambda u_2 & -(\lambda^2+1)u_4-\lambda u_1 & 0  \\

    \end{pmatrix}}
      \end{array}\]

  \label{tab2}
\end{table}

%{\large{\bf Table 3: Matrices associated to forms of rank $7$.}}

  \begin{table}[ht]
    \caption{Matrices associated to forms of rank $7$.}
    \[\arraycolsep=1pt
      \begin{array}[t]{l|l}
        \text{Type\ } & \text{Matrix} \\ \hline
        T_5 & {\tiny \begin{pmatrix}
      0 & u_3 & -u_2 & u_7 & 0 & 0 & -u_4 \\
      -u_3 & 0 & u_1 & 0 & 0 & 0 & 0 \\
      u_2 & -u_1 & 0 & 0 & 0 & 0 & 0 \\
      -u_7 & 0   & 0    & 0 & u_6 & -u_5 & u_1 \\
      0 & 0   & 0    & -u_6 & 0 & u_4 & 0 \\
      0 & 0   & 0    & u_5 & -u_4 & 0 & 0 \\
      u_4 & 0   & 0    & -u_1 & 0 & 0 & 0 \\
    \end{pmatrix} } \\
        T_6 & {\tiny
                  \begin{pmatrix}
      0 & -u_5 & -u_6 & -u_7 & u_2 & u_3 & u_4 \\
      u_5 & 0 & -u_4 & u_3 & -u_1 & 0 & 0 \\
      u_6 & u_4 & 0 & -u_2 & 0 & -u_1 & 0 \\
      u_7 & -u_3   & u_2    & 0 & 0 & 0 & -u_1 \\
      -u_2 & u_1   & 0    & 0 & 0 & 0 & 0 \\
      -u_3 & 0   & u_1    & 0 & 0 & 0 & 0 \\
      -u_4 & 0   & 0    & u_1 & 0 & 0 & 0 \\
    \end{pmatrix}} \\
        T_7 & {\tiny
                  \begin{pmatrix}
      0 & 0 & 0 & u_6 & u_7 & -u_4 & -u_5 \\
      0 & 0 & 0 & u_5 & -u_4 & 0 & 0 \\
      0 & 0 & 0 & 0 & 0 & u_7 & -u_6 \\
      -u_6 & -u_5   & 0    & 0 & u_2 & u_1 & 0 \\
      -u_7 & u_4   & 0    & -u_2 & 0 & 0 & u_1 \\
      u_4 & 0   & -u_7    & -u_1 & 0 & 0 & u_3 \\
      u_5 & 0   & u_6   & 0 & -u_1 & -u_3 & 0 \\
    \end{pmatrix}} \\
      \end{array}\qquad
      \begin{array}[t]{l|l}
        \text{Type\ } & \text{\ Matrix} \\ \hline
                T_8 & {\tiny
                  \begin{pmatrix}
      0 & u_3 & -u_2 & u_5 & -u_4 & u_7 & -u_6 \\
      -u_3 & 0 & u_1 & 0 & 0 & 0 & 0 \\
      u_2 & -u_1 & 0 & 0 & 0 & 0 & 0 \\
      -u_5 & 0  & 0    & 0 & u_1 & 0 & 0 \\
      u_4 & 0   & 0    & -u_1 & 0 & 0 & 0 \\
      -u_7 & 0   & 0    & 0 & 0 & 0 & u_1 \\
      u_6 & 0   & 0   & 0 & 0 & -u_1 & 0 \\
    \end{pmatrix}} \\
        T_9 & {\tiny
              \begin{pmatrix}
                0 & u_3 & -u_2 & u_7 & 0 & 0 & -u_4 \\
                -u_3 & 0 & u_1 &0 & u_7 & 0 & -u_5 \\
                u_2 & -u_1 & 0 & 0 & 0 & u_7 & -u_6 \\
                -u_7 & 0 & 0 & 0 & u_6 & -u_5 & u_1 \\
                0 & -u_7 & 0 & -u_6 & 0 & u_4 & u_2 \\
                0 & 0 & -u_7 & u_5 & -u_4 & 0 & u_3 \\
                u_4 & u_5 & u_6 & -u_1 & -u_2 & -u_3 & 0
              \end{pmatrix}} \\
                T_{11,\lambda}^{(1)} & {\tiny
                               \begin{pmatrix}

      0 & u_3 & -u_2 & u_7 & \lambda u_6 & -\lambda u_5 &  -u_4 \\
      -u_3 & 0 & u_1 & -\lambda u_6 & 0 & \lambda u_4 & 0 \\
      u_2 & -u_1 & 0 & \lambda u_5 & -\lambda u_4 & 0 & 0 \\
      -u_7 & \lambda u_6 & -\lambda u_5 & 0 & \lambda u_3 & -\lambda u_2 & u_1 \\
      -\lambda u_6 & 0 & \lambda u_4 & -\lambda u_3 & 0 & \lambda u_1 & 0 \\
      \lambda u_5 & -\lambda u_4 & 0 & \lambda u_2 & -\lambda u_1 & 0 & 0 \\
      u_4 & 0 & 0 & -u_1 & 0 & 0 & 0 \\
    \end{pmatrix}} \\
      \end{array} \]
    \[
      \arraycolsep=1pt
      \begin{array}[t]{l|l}
        \text{Type\ } & \text{Matrix} \\ \hline
        T_{11,\lambda}^{(2)} & {\tiny     \begin{pmatrix}
      0 & u_6 & -u_5 & u_7 & \lambda u_6+u_3 & -\lambda u_5-u_2 &  -u_4 \\
      -u_6 & 0 & u_4 & -\lambda u_6-u_3 & 0 & \lambda u_4+u_1 & 0 \\
      u_5 & -u_4 & 0 & \lambda u_5+u_2 & -\lambda u_4-u_1 & 0 & 0 \\
      -u_7 & \lambda u_6+u_3 & -\lambda u_5-u_2 & 0 & (\lambda^2+1)u_6+\lambda u_3 & (-\lambda^2-1)u_5 -\lambda u_2 & u_1 \\
      -\lambda u_6-u_3 & 0 & \lambda u_4+u_1 & -(\lambda^2+1)u_6-\lambda u_3 & 0 & (\lambda^2+1)u_4+\lambda u_1 & 0 \\
      \lambda u_5+u_2 & -\lambda u_4-u_1 & 0 & (\lambda^2+1)u_5+\lambda u_2 & -(\lambda^2+1)u_4-\lambda u_1 & 0 & 0 \\
      u_4 & 0 & 0 & -u_1 & 0 & 0 & 0 \\
    \end{pmatrix} }
      \end{array}
    \]
    \label{tab3}
  \end{table}
%{\large{\bf Table 4: Description of the geometry of poles associated to forms of rank up to $6$.}}

\begin{table}
    \caption{Description of the geometries of poles  associated to forms of rank up to $6$.\hskip-1cm}
    \[\arraycolsep=1pt
      \begin{array}[t]{c|c|l}
        \text{Type\ } & \text{Poles} & \multicolumn{1}{c}{\text{Upper radical}}  \\ \hline
        T_1 & \PG(V) & |x,y|_{12}=0, |x,y|_{13}=0, |x,y|_{23}=0.  \\[2pt]
        T_2 & \PG(V) & \begin{array}[t]{l}
                         |x,y|_{12}=0, |x,y|_{13}=0, |x,y|_{15}=0, \\
                         |x,y|_{14}=0, |x,y|_{23}+|x,y|_{45}=0.
                       \end{array} \\[2pt]
        T_3 & \PG(V) & \begin{array}[t]{l}
                         |x,y|_{12}=0, |x,y|_{13}=0, |x,y|_{23}=0, \\
                         |x,y|_{45}=0, |x,y|_{46}=0, |x,y|_{56}=0.
                  \end{array} \\[2pt]
        T_4 & \PG(V) & \begin{array}[t]{l}
                    |x,y|_{26}-|x,y|_{35}=0, |x,y|_{16}-|x,y|_{34}=0 \\
                    |x,y|_{24}-|x,y|_{15}=0, |x,y|_{23}=0, \\
                    |x,y|_{13}=0, |x,y|_{12}=0.
                  \end{array} \\[2pt]
        T_{10,\lambda}^{(1)} & \PG(V) & \begin{array}[t]{l}
                                          |x,y|_{23}+\lambda|x,y|_{56}=0, |x,y|_{26}-|x,y|_{35}=0 \\
                                          |x,y|_{13}+\lambda|x,y|_{46}=0, |x,y|_{16}-|x,y|_{34}=0 \\
                                          |x,y|_{12}+\lambda|x,y|_{45}=0, |x,y|_{15}-|x,y|_{24}=0
                                          \end{array} \\[2pt]
        T_{10,\lambda}^{(2)} & \PG(V) & \begin{array}[t]{l}
                                          |x,y|_{26}-|x,y|_{35}+\lambda|x,y|_{56}=0, \\ -|x,y|_{16}+|x,y|_{34}-\lambda|x,y|_{46}=0, \\
                                          |x,y|_{15}-|x,y|_{24}+\lambda|x,y|_{45}=0, \\
                                          |x,y|_{23}+\lambda|x,y|_{26}-\lambda|x,y|_{35}+(\lambda^2+1)|x,y|_{56}=0,\\
                                          -|x,y|_{13}-\lambda|x,y|_{16}+\lambda|x,y|_{34}-(1+\lambda^2)|x,y|_{46}=0,\\
                                          |x,y|_{12}+\lambda|x,y|_{15}-\lambda|x,y|_{24}+(1+\lambda^2)|x,y|_{45}=0.
                                          \end{array}
      \end{array}
    \]
    \label{tab_n6}
\end{table}
%{\large{\bf Table 4: Description of the geometry of poles associated to forms of rank $7$.}}

\begin{table}[ht]
    \caption{Description of the geometry of poles associated to forms of rank $7$.}
    \[\arraycolsep=1pt
      \begin{array}[t]{c|c|l}
        \text{Type\ } & \text{Poles} & \multicolumn{1}{c}{\text{Upper radical}}  \\ \hline
        T_5 & {x}_1{x}_4=0 & \begin{array}[t]{l}
                           |x,y|_{23}+|x,y|_{47}=0, |x,y|_{13}=0, |x,y|_{12}=0, \\
                           |x,y|_{56}-|x,y|_{17}=0, |x,y|_{14}=0, |x,y|_{45}=0, \\
                           |x,y|_{46}=0
                         \end{array} \\[2pt]
        T_6 & {x}_1^2=0 & \begin{array}[t]{l}
                          |x,y|_{25}+|x,y|_{36}+|x,y|_{47}=0, |x,y|_{14}=0, \\ |x,y|_{15}-|x,y|_{34}=0,
                          |x,y|_{16}+|x,y|_{24}=0, \\ |x,y|_{17}-|x,y|_{23}=0,
                          |x,y|_{12}=0, |x,y|_{13}=0.
                        \end{array} \\[2pt]
        T_7 & {x}_5{x}_7+{x}_4{x}_6=0 &
                                \begin{array}[t]{l}
                                  |x,y|_{46}+|x,y|_{57}=0, |x,y|_{45}=0, |x,y|_{67}=0, \\
                                  |x,y|_{16}+|x,y|_{25}=0, |x,y|_{24}-|x,y|_{17}=0, \\
                                  |x,y|_{14}-|x,y|_{37}=0, |x,y|_{15}+|x,y|_{36}=0.
                                \end{array} \\[2pt]
        T_8 & {x}_1^2=0 & \begin{array}[t]{l}
                          |x,y|_{23}+|x,y|_{45}+|x,y|_{67}=0, |x,y|_{13}=0, \\ |x,y|_{12}=0,
                          |x,y|_{14}=0, |x,y|_{15}=0, \\ |x,y|_{16}=0, |x,y|_{17}=0.
                        \end{array} \\[2pt]
        T_9 & {x}_7^2-{x}_3{x}_6-{x}_2{x}_5-{x}_1{x}_4=0 & \begin{array}[t]{l}
                                               |x,y|_{23}+|x,y|_{47}=0, |x,y|_{57}-|x,y|_{13}=0 \\
                                               |x,y|_{12}+|x,y|_{67}=0, |x,y|_{56}-|x,y|_{17}=0 \\
                                               |x,y|_{27}+|x,y|_{46}=0, |x,y|_{45}-|x,y|_{37}=0 \\
                                               |x,y|_{14}+|x,y|_{25}+|x,y|_{36}=0.
                                               \end{array} \\[2pt]
        T_{11,\lambda}^{(1)} & \lambda {x}_4^2-{x}_1^2=0 & \begin{array}[t]{l}
                                          |x,y|_{13}+\lambda|x,y|_{46}=0, |x,y|_{12}+\lambda|x,y|_{45}=0 \\
                                          |x,y|_{23}+|x,y|_{47}+\lambda|x,y|_{56}=0,  |x,y|_{14}=0, \\
                                          -|x,y|_{17}+\lambda(|x,y|_{26}-|x,y|_{35})=0, \\
                                          |x,y|_{15}-|x,y|_{24}=0,  |x,y|_{16}-|x,y|_{34}=0. \\
                                          \end{array} \\[2pt]
        T_{11,\lambda}^{(2)} &{x}_4^2+\lambda {x}_1{x}_4+{x}_1^2=0 & \begin{array}[t]{l}
                                                               |x,y|_{26}-|x,y|_{35}+|x,y|_{47}+\lambda |x,y|_{56}=0, \\
                                                               |x,y|_{16}-|x,y|_{34}+\lambda|x,y|_{46}=0, |x,y|_{14}=0, \\
                                                               |x,y|_{15}-|x,y|_{24}+\lambda|x,y|_{45}=0, \\
                                                               |x,y|_{17}-|x,y|_{23}-\lambda(|x,y|_{26}-|x,y|_{35})-\\
                         \multicolumn{1}{r}{-(\lambda^2+1)|x,y|_{56}=0,} \\
|x,y|_{13}+\lambda(|x,y|_{16}-|x,y|_{34})+\\
                         \multicolumn{1}{r}{+(\lambda^2+1)|x,y|_{46}=0,} \\
|x,y|_{12}+\lambda|x,y|_{15}-\lambda|x,y|_{24}+\\
                         \multicolumn{1}{r}{+(\lambda^2+1)|x,y|_{45}=0.}
                                                               \end{array}
      \end{array} \]
    \label{tab_n7}
  \end{table}

\end{document}